\documentclass[11pt, a4paper]{amsart}
\usepackage{amsmath}
\usepackage{amsthm}
\usepackage{amssymb}
\usepackage{mathrsfs}
\usepackage[sort,numbers]{natbib}
\usepackage{xcolor}
\definecolor{Mycolor1}{HTML}{32B2B2}
\usepackage[colorlinks=true]{hyperref}
\hypersetup{
linkcolor=blue, urlcolor=blue, citecolor=blue, pdfstartview={FitH}, hyperfootnotes=true, unicode=true}
\numberwithin{equation}{section}

\newcommand{\dif}{\textrm{\normalfont d}}
\newcommand{\re}{\textrm{\normalfont Re}\,}
\newcommand{\diag}{\textrm{\normalfont diag}\,}

\newcommand{\ran}{\textrm{\normalfont ran}\,}
\newcommand{\trace}{\textrm{\normalfont tr}\,}
\newcommand{\diver}{\textrm{\normalfont div}\,}

\DeclareMathOperator*{\esssup}{ess\,sup}

\newtheorem{theorem}{Theorem}[section]
\newtheorem{proposition}[theorem]{Proposition}
\newtheorem{lemma}[theorem]{Lemma}
\newtheorem{corollary}[theorem]{Corollary}

\theoremstyle{definition}
\newtheorem{definition}[theorem]{Definition}
\newtheorem{remark}[theorem]{Remark}
\newtheorem{example}[theorem]{Example}

\begin{document}
\begin{center}
ASYMPTOTIC LIMIT AND DECAY ESTIMATES FOR A CLASS OF\\
DISSIPATIVE LINEAR HYPERBOLIC SYSTEMS\\
IN SEVERAL DIMENSIONS

\vskip.25cm
%{\small\tt LpLqmultiD.tex} -- \today\vskip.2cm
\today\vskip.2cm
Thinh Tien NGUYEN\footnote{Department of Mathematics, Gran Sasso Science Institute -- Istituto Nazionale di Fisica Nucleare,
Viale Francesco Crispi, 7 - 67100 L'Aquila (ITALY),  \texttt{\tiny nguyen.tienthinh@gssi.infn.it}}
\end{center}
\vskip.5cm

\begin{quote}\footnotesize\baselineskip 12pt 
{\sf Abstract.}  
In this paper, we study the large-time behavior of solutions to a class of partially dissipative linear hyperbolic systems with applications in velocity-jump processes in several dimensions. Given integers $n,d\ge 1$, let $\mathbf A:=(A^1,\dots,A^d)\in (\mathbb R^{n\times n})^d$ be a matrix-vector, where $A^j\in\mathbb R^{n\times n}$, and let $B\in \mathbb R^{n\times n}$ be not required to be symmetric but have one single eigenvalue zero, we consider the Cauchy problem for linear $n\times n$ systems having the form
\begin{equation*}
	\partial_{t}u+\mathbf A\cdot \nabla_{\mathbf x} u+Bu=0,\qquad (\mathbf x,t)\in \mathbb R^d\times \mathbb R_+.
\end{equation*}
Under appropriate assumptions, we show that the solution $u$ is decomposed into $u=u^{(1)}+u^{(2)}$, where $u^{(1)}$ has the asymptotic profile which is the solution, denoted by $U$, of a parabolic equation and $u^{(1)}-U$ decays at the rate $t^{-\frac d2(\frac 1q-\frac 1p)-\frac 12}$ as $t\to +\infty$ in any $L^p$-norm, and $u^{(2)}$ decays exponentially in $L^2$-norm, provided $u(\cdot,0)\in L^q(\mathbb R^d)\cap L^2(\mathbb R^d)$ for $1\le q\le p\le \infty$. Moreover, $u^{(1)}-U$ decays at the optimal rate $t^{-\frac d2(\frac 1q-\frac 1p)-1}$ as $t\to +\infty$ if the system satisfies a symmetry property. The main proofs are based on asymptotic expansions of the solution $u$ in the frequency space and the Fourier analysis.
\vskip.15cm

\vskip.15cm

\end{quote}

%\tableofcontents
\section{Introduction}
Consider the Cauchy problem for partially dissipative linear hyperbolic systems
\begin{equation}\label{eq:the dissipative hyperbolic equation}
\begin{cases}
	\partial_tu+\mathbf A\cdot \nabla_{\mathbf x}u+Bu=0,&(\mathbf x,t)\in \mathbb R^d\times \mathbb R_+,\\
	u(\mathbf x,0)=u_0(\mathbf x),
\end{cases}
\end{equation}
where $\mathbf A=(A^1,\dots,A^d)\in (\mathbb R^{n\times n})^d$ and $B\in\mathbb R^{n\times n}$, not required to be symmetric. The system \eqref{eq:the dissipative hyperbolic equation} can be regarded as discrete-velocity models where $\mathbf A$ determines the velocities of moving particles and $B$ gives the transition rates of the velocities after collisions among the particles in the system. For instance, this type of dissipative linear systems arises in the Goldstein--Kac model \cite{goldstein51,kac74} and the model of neurofilament transport in axons \cite{craciun05}.

The large-time behavior of the solution $u$ to \eqref{eq:the dissipative hyperbolic equation} in terms of decay estimates has been established for years. It follows from \cite{ueda12} that under appropriate assumptions on $\mathbf A$ and $B$, if $u$ is the solution to \eqref{eq:the dissipative hyperbolic equation} with the initial data $u_0\in L^1(\mathbb R^d)\cap L^2(\mathbb R^d)$, then one has
\begin{equation}\label{eq:ueda estimate}
	\|u\|_{L^2}\le C(1+t)^{-\frac d4}\|u_0\|_{L^1}+Ce^{-ct}\|u_0\|_{L^2},\qquad \forall t>0,
\end{equation}
for some positive constants $c$ and $C$. Moreover, the estimate \eqref{eq:ueda estimate} was generalized in \cite{bianchini07}, where $B$ can be written in the conservative-dissipative form $B=\diag(O,D)$ with $D$, a positive definite matrix not required to be symmetric. The authors in \cite{bianchini07} also showed that the conservative part $u^{(1)}$, describing the low dynamics of $u$ to the equilibrium manifold $\ker B$, satisfies the estimate
\begin{equation}\label{eq:bianchini estimate}
	\|u^{(1)}-U\|_{L^p}\le Ct^{-\frac d2(1-\frac 1p)-\frac 12}\|u_0\|_{L^1}, \qquad \forall p\ge \min\{d,2\}, t\ge 1,
\end{equation}
while the dissipative part $u^{(2)}$ of $u$ decays exponentially in $L^2(\mathbb R^d)$, where $U$ is the solution to a parabolic system given by applying the Chapman--Enskog expansion method. A more exact asymptotic parabolic-limit $U$ of $u$ was established for a class of the generalized Goldstein--Kac system in \cite{mascia16}, where the matrix $B$ is symmetric. The solution $u$ to \eqref{eq:the dissipative hyperbolic equation} then satisfies the estimate
\begin{equation}\label{eq:mascia estimate}
	\|u-U\|_{L^2}\le Ct^{-\frac d4-\frac 12}\|u_0\|_{L^1\cap L^2}, \qquad \forall t\ge 1,
\end{equation}
where $U$ is obtained based on an exhaustive analysis of the dispersion relation and on the application of a variant of the Kirchoff's matrix tree theorem from the graph theory. Recently, \cite{mascianguyen16} has optimized the above decay estimates for the solution $u$ to \eqref{eq:the dissipative hyperbolic equation} in one dimension $d=1$, with an explicit parabolic-limit $U$ and a corrector $V$, namely, one has
\begin{equation}\label{eq:mascianguyen estimate}
	\|u-U-V\|_{L^p}\le Ct^{-\frac 12(\frac 1q-\frac 1p)-\delta}\|u_0\|_{L^q},\qquad \forall 1\le q\le p\le \infty, t\ge 1,
\end{equation}
where $\delta \in \{1/2,1\}$, $U$ solving a parabolic system arising in the low-frequency analysis decays diffusively, and $V$ solving a hyperbolic system arising in the high-frequency analysis decays exponentially. The decay estimate \eqref{eq:mascianguyen estimate} is remarkable since it holds for general $p$ and $q$ ranging over $[1,\infty]$. Such kind of decay estimates is very well-known {\it e.g.} the $L^p$-$L^q$ decay estimate for the linear damped wave equation as in \cite{marcati03,nishihara03,hosono04,narazaki04}.

To obtain \eqref{eq:mascianguyen estimate} in one dimension $d=1$, one primarily considers the asymptotic expansions of the fundamental solution to the system \eqref{eq:the dissipative hyperbolic equation} in the Fourier space, divided into the low frequency, the intermediate frequency and the high frequency which naturally produce the time-asymptotic profile. Then, by an interpolation argument once the $L^\infty$-$L^1$ estimate and the $L^p$-$L^p$ estimate for $1\le p\le \infty$ are accomplished, one obtains the desired $L^p$-$L^q$ estimate for any $1\le q\le p\le \infty$. The same strategy will be applied to the system \eqref{eq:the dissipative hyperbolic equation} in several dimensions $d\ge 2$ in this paper. Nevertheless, difficulties occur as the dimension $d$ increases. For instance, as mentioned in \cite{bianchini07}, one cannot expect the estimate
\begin{equation}\label{eq:L1-L1}
	\|u\|_{L^1}\le C\|u_0\|_{L^1}
\end{equation}
hold in general since for large time, $L_0u$, where $L_0$ is the left eigenvector associated with the eigenvalue $0$ of $B$, behaves as the solution $\omega$ to the reduced system
\begin{equation}\label{eq:reduced system}
	\partial_t\omega+L_0\mathbf AR_0\cdot \nabla_{\mathbf x} \omega=0,\qquad (\mathbf x,t)\in\mathbb R^d\times \mathbb R_+,
\end{equation}
where $R_0$ is the right eigenvector associated with the eigenvalue $0$ of $B$, and thus, it is known in \cite{bressan03} that \eqref{eq:L1-L1} is not true in general. The estimate \eqref{eq:L1-L1} in fact depends strongly on a uniform parabolic operator. Nonetheless, this obstacle can be defeated if $d=1$ as in \cite{mascianguyen16} or if $0$ is a simple eigenvalue of $B$ since the system \eqref{eq:reduced system} then becomes scalar and it allows us to obtain \eqref{eq:L1-L1} as we will see in this paper. Another difficulty arises in the high-frequency analysis due to the loss of integrability and the fact that one cannot perform a uniform expansion of the fundamental solution as the dimension $d$ increases. Hence, the corrector $V$ as in \eqref{eq:mascianguyen estimate} cannot be obtained trivially.

The aim of this paper is to study the $L^p$-$L^q$ decay estimate for the conservative part $u^{(1)}$ of the solution $u$ to the system \eqref{eq:the dissipative hyperbolic equation} in several dimensions $d\ge 2$ for general $p$ and $q$ in $[1,\infty]$ in order to generalize \eqref{eq:ueda estimate}, \eqref{eq:bianchini estimate} and \eqref{eq:mascia estimate}, where $B$ is not required to be symmetric but has one single eigenvalue zero. The $L^p$-$L^q$ estimate as in \eqref{eq:mascianguyen estimate} for the multi-dimensional case $d\ge 2$ is still a challenge for the author.

\vskip.25cm
For $\mathbf x=(x_1,\dots,x_d)\in\mathbb C^d$, consider the $n\times n$ operators
\begin{equation}\label{eq:operators E and A}
	E(\mathbf x):=B+A(\mathbf x),\qquad A(\mathbf x):=\mathbf A\cdot \mathbf x=\sum_{j=1}^dA^jx_j,
\end{equation}
where $\mathbf A=(A^1,\dots,A^d)\in (\mathbb R^{n\times n})^d$ and $B\in \mathbb R^{n\times n}$. We start with the following reasonable assumptions.
\vskip.25cm
{Condition $\mathsf A$.} {\sf [Hyperbolicity]} {\it $A=A(\mathbf w)$ for $\mathbf w\in\mathbb S^{d-1}$ is uniformly diagonalizable with real linear eigenvalues {\it i.e.} there is an invertible matrix $R=R(\mathbf w)$ for $\mathbf w\in\mathbb S^{d-1}$ satisfying
\begin{equation*}
	\sup_{\mathbf w\in\mathbb S^{d-1}}|R(\mathbf w)||R^{-1}(\mathbf w)|<+\infty,
\end{equation*}
for any matrix norm, such that $R^{-1}AR$ is a diagonal matrix whose nonzero entries are real linear in $\mathbf w\in\mathbb S^{d-1}$.}

\vskip.25cm
{Condition $\mathsf R$.} {\sf [Diagonalizing matrix]} {\it There is a matrix $R$ uniformly diagonalizing $A$ such that $R^{-1}BR$ is a constant matrix.}

\vskip.25cm
{Condition $\mathsf B$.} {\sf [Partial dissipation]} {\it The spectrum of $B$ is decomposed into $\sigma(B)=\{0\}\cup \sigma_0$ where $0$ is simple and $\sigma_0\subseteq \{z\in \mathbb C:\re z>0\}$.}

\vskip.25cm
Moreover, the requisite condition for the decay of the solution $u$ to \eqref{eq:the dissipative hyperbolic equation}, strictly related to the Shizuta--Kawashima condition: {\it the eigenvectors of $A(\mathbf x)$ do not belong to the kernel of $B$ for any $\mathbf x\ne \mathbf 0$} (see \cite{kawashima84,shizuta85,umeda84} and therein), is given by
\vskip.25cm
{Condition $\mathsf D$.} {\sf [Uniform dissipation]} {\it There is a constant $\theta>0$ such that for any eigenvalue $\lambda=\lambda(i\mathbf k)$ of $E=E(i\mathbf k)$ in \eqref{eq:operators E and A} for $\mathbf k\in\mathbb R^d$, one has
\begin{equation*}
	\re\lambda(i\mathbf k)\ge \dfrac{\theta|\mathbf k|^2}{1+|\mathbf k|^2},\qquad \forall \mathbf k\ne \mathbf 0 \in \mathbb R^d.
\end{equation*}}

\begin{remark}[Relaxing the conditions $\mathsf A$ and $\mathsf R$]
The requirement of the linearity of the eigenvalues of the matrix $A$ satisfying the condition $\mathsf A$ and the existence of the matrix $R$ satisfying the condition $\mathsf R$ can be omitted by considering the dissipative structures proposed in \cite{ueda12,bianchini07}. Nonetheless, the structures in \cite{ueda12,bianchini07} require that the system \eqref{eq:the dissipative hyperbolic equation} is Friedrich symmetrizable while in our case, the matrix $A$ is only uniformly diagonalizable. The advantage of the linearity of the eigenvalues of the matrix $A$ and the existence of the matrix $R$ is that one can construct the high-frequency asymptotic expansion of $E$ in \eqref{eq:operators E and A} after subtracting a suitable Lebesgue measure zero set.
\end{remark}

We now construct the asymptotic parabolic-limit $U$ of the solution $u$ to \eqref{eq:the dissipative hyperbolic equation}. Let $\Gamma$ be an oriented closed curve in the resolvent set of $B$ such that it encloses zero except for the other eigenvalues of $B$. One sets
\begin{equation}\label{eq:coefficients P0 and S0 1}
	P_0^{(0)}:=-\dfrac{1}{2\pi i}\int_{\Gamma}(B-zI)^{-1}\,\dif z,\qquad Q_0^{(0)}:=\dfrac{1}{2\pi i}\int_{\Gamma}z^{-1}(B-zI)^{-1}\,\dif z,
\end{equation}
which are the eigenprojection and the reduced resolvent coefficient associated with the eigenvalue zero of $B$. We consider the Cauchy problem
\begin{equation}\label{eq:the parabolic equation}
\begin{cases}
	\partial_tU+\mathbf c\cdot \nabla_{\mathbf x}U-\diver(\mathbf D\nabla_{\mathbf x}U)=0,&(\mathbf x,t)\in \mathbb R^d\times \mathbb R_+,\\
	U(\mathbf x,0)=P_0^{(0)}u_0(\mathbf x),
\end{cases}
\end{equation}
where $\mathbf c=(c_h)\in \mathbb R^d$ and $\mathbf D=(D_{h\ell}) \in \mathbb R^{d\times d}$ is positive definite with scalar entries
\begin{equation}\label{eq:coefficients c and D 1}
	c_h:=\trace\bigl(A^h P_0^{(0)}\bigr),\qquad D_{h\ell}:=\dfrac{1}{2}\trace\bigl(A^{h}P_0^{(0)}A^\ell Q_0^{(0)}+A^h Q_0^{(0)}A^\ell P_0^{(0)}\bigr).
\end{equation}

\begin{theorem}[$L^p$-$L^q$ decay estimates]\label{theo:LpLq estimates 1}
Let $u$ be the solution to the Cauchy problem \eqref{eq:the dissipative hyperbolic equation} with the initial data $u_0\in L^q(\mathbb R^d)\cap L^2(\mathbb R^d)$ for $1\le q\le \infty$. Under the assumptions $\mathsf A$, $\mathsf R$, $\mathsf B$ and $\mathsf D$, the solution $u$ is decomposed into
\begin{equation}\label{eq:the decomposition of u}
	u(\mathbf x,t)=u^{(1)}(\mathbf x,t)+u^{(2)}(\mathbf x,t),
\end{equation}
where
\begin{equation*}
	u^{(1)}(\mathbf x,t):=\mathcal F^{-1}(e^{-E(i\mathbf k)t}P_0(i\mathbf k)\chi(\mathbf k))*u_0(\mathbf x)
\end{equation*}
and $u^{(2)}$ is the remainder, where $P_0$ is the eigenprojection associated with the eigenvalue of $E$ in \eqref{eq:operators E and A} converging to $0$ as $|\mathbf k|\to 0$ and $\chi$ is a cut-off function with support contained in the ball $B(\mathbf 0,\varepsilon)\subset \mathbb R^d$, valued in $[0,1]$, for small $\varepsilon>0$. Moreover, for any $1\le q\le p\le \infty$ and $t\ge 1$, one has
\begin{equation}\label{eq:decay estimates 1}
	\|u^{(1)}-U\|_{L^p}\le Ct^{-\frac d2(\frac 1q-\frac 1p)-\frac 12}\|u_0\|_{L^q},
\end{equation}
where $U$ is the solution to \eqref{eq:the parabolic equation} with the initial data $U_0\in L^q(\mathbb R^d)$, and one has
\begin{equation}\label{eq:decay estimates 2}
	\|u^{(2)}\|_{L^2}\le Ce^{-ct}\|u_0\|_{L^2}
\end{equation}
for some constant $c>0$ and for all $t\ge 1$.
\end{theorem}
\begin{remark}[Finite speed of propagation]
In the case where the solution $u$ to the system \eqref{eq:the dissipative hyperbolic equation} has finite speed of propagation, since the fundamental solution associated with $u$ has compact support contained in the wave cone $\{(\mathbf x,t)\in\mathbb R^d\times \mathbb R:|\mathbf x/t|\le C\}$ for some constant $C>0$, one can decompose $u$ into $u=u^{(1)}+u^{(2)}$, where
\begin{equation*}
	u^{(1)}(\mathbf x,t):=\mathcal F^{-1}(e^{-E(i\mathbf k)t}\chi(\mathbf k))*u_0(\mathbf x),
\end{equation*}
and $u^{(2)}$ is the remainder, where $\chi$ is a cut-off function with support contained in the ball $B(\mathbf 0,\rho)\subset \mathbb R^d$, valued in $[0,1]$, for any $\rho>0$, and the estimates \eqref{eq:decay estimates 1} and \eqref{eq:decay estimates 2} still hold for $t\ge 1$. This fact will be proved in the subsequent sections. For instance, it is the case where the system \eqref{eq:the dissipative hyperbolic equation} is Friedrich symmetrizable. Nonetheless, in one dimension $d=1$, the case $|\mathbf x/t|>C$ can be treated since the Cauchy integral theorem holds for the whole complex plane, and thus, one can use the estimates for the asymptotic expansion of the fundamental solution in the high frequency after changing paths of integrals of holomorphic functions (see \cite{mascianguyen16}).
\end{remark}

Moreover, consider the one-dimensional $2\times 2$ linear Goldstein--Kac system
\begin{equation*}
	\begin{cases}
		\partial_tu_1-\partial_xu_1=-\frac 12 u_1+\frac 12 u_2,\\
		\partial_tu_2+\partial_xu_2=\frac 12 u_1-\frac 12 u_2,
	\end{cases}\quad (x,t)\in\mathbb R\times \mathbb R_+.
\end{equation*}
It can be checked easily that $w:=u_1+u_2$ satisfies the linear damped wave equation
\begin{equation*}
	\begin{cases}
		\partial_{tt}^2w-\partial_{xx}^2w+\partial_tw=0,&(x,t)\in\mathbb R\times \mathbb R_+,\\
		w(x,0)=w_0(x),\\
		\partial_tw(x,0)=w_1(x),
	\end{cases}
\end{equation*}
where $w_0$ and $w_1$ are appropriate initial data. It then follows from \cite{marcati03} that
\begin{equation}\label{eq:LpLq estimates in marcati03}
	\Bigl\|w-\phi-e^{-\frac t2}\dfrac{w_0(x+t)+w_0(x-t)}{2}\Bigr\|_{L^p}\le Ct^{-\frac 12(\frac 1q-\frac 1p)-1}\|(w_0,w_1)\|_{L^q},
\end{equation}
for any $1\le q\le p\le \infty$ and $t\ge 1$, where $\phi$ is the solution to the heat equation
\begin{equation*}
	\begin{cases}
		\partial_t\phi-\partial_{xx}^2\phi=0,&(x,t)\in \mathbb R\times \mathbb R_+,\\
		\phi(x,0)=w_0(x)+w_1(x).
	\end{cases}
\end{equation*}
Without regarding the exponentially decaying term in \eqref{eq:LpLq estimates in marcati03}, there is a difference of a quantity of $1/2$ between the decay rates \eqref{eq:LpLq estimates in marcati03} and \eqref{eq:decay estimates 1}. The difference can be explained by a symmetry property that the one-dimensional $2\times 2$ linear Goldstein--Kac system possesses. Such kind of symmetry properties is already studied in \cite{mascianguyen16} based on the existence of an invertible matrix $S$ commuting with $B$ and anti-commuting with the matrix $\mathbf A=A$ of one-dimensional dissipative linear hyperbolic systems. More general, in several dimensions $d\ge 2$, the symmetry property is given by
\vskip.25cm
{Condition $\mathsf S$.} {\sf [Symmetry]} {\it There is an invertible matrix $S=S(\mathbf w)$ for $\mathbf w\in\mathbb S^{d-1}$ such that
\begin{equation*}
	SB=BS,\qquad SA=-AS,
\end{equation*}}
where $A=A(\mathbf w)$ is given by \eqref{eq:operators E and A} for $\mathbf w\in\mathbb S^{d-1}$.

\begin{remark}
In practice, it is easy to check the condition $\mathsf S$ if there is a constant invertible matrix $S$ satisfying $SB=BS$ and $SA^j=-A^jS$ for all $j\in\{1,\dots,d\}$ since $A(\mathbf w)=\sum_{j=1}^dA^j w_j$ by definition.
\end{remark}
\vskip.25cm
We will show that under the conditions $\mathsf B$, $\mathsf D$ and $\mathsf S$, the decay rate in the estimate \eqref{eq:decay estimates 1} increases. We primarily refine the asymptotic profile $U$.

With the coefficients $P_0^{(0)}, \,Q_0^{(0)}$ in \eqref{eq:coefficients P0 and S0 1} and $\mathbf D$ in \eqref{eq:coefficients c and D 1}, we consider the Cauchy problem
\begin{equation}\label{eq:the parabolic equation 2}
\begin{cases}
	\partial_tU-\diver(\mathbf D\nabla_{\mathbf x}U)=0,&(\mathbf x,t)\in \mathbb R^d\times \mathbb R_+,\\
	U(\mathbf x,0)=P_0^{(0)}u_0(\mathbf x)+\mathbf P_0^{(1)}\cdot \nabla_{\mathbf x}u_0(\mathbf x),
\end{cases}
\end{equation}
where $\mathbf P_0^{(1)}=(P_{0h}^{(1)})\in(\mathbb R^{n\times n})^d$ with matrix entries
\begin{equation}\label{eq:coefficient P01 of the total projection P0 1}
	P_{0h}^{(1)}:=-P_0^{(0)}A^hQ_0^{(0)}-Q_0^{(0)}A^hP_0^{(0)}.
\end{equation}
\begin{theorem}[Optimal decay rate]\label{theo:LpLq estimates 2}
Under the same hypotheses of Theorem \ref{theo:LpLq estimates 1}, if the condition $\mathsf S$ holds in addition, the solution $u$ is also decomposed into $u=u^{(1)}+u^{(2)}$ as in \eqref{eq:the decomposition of u} such that
 for any $1\le q\le p\le \infty$ and $t\ge 1$, one has
\begin{equation}\label{eq:decay estimates increased}
	\|u^{(1)}-U\|_{L^p}\le Ct^{-\frac d2(\frac 1q-\frac 1p)-1}\|u_0\|_{L^q},
\end{equation}
where $U$ is the solution to \eqref{eq:the parabolic equation 2} with the initial data $U_0$.
\end{theorem}

The paper is organized as follows. Section \ref{sec:Proofs of the main theorems} is devoted to proofs and examples of Theorem \ref{theo:LpLq estimates 1} and Theorem \ref{theo:LpLq estimates 2}, where the proofs are based on the estimates obtained in Section \ref{sec:Decay estimates}. In order to prove these estimates in Section \ref{sec:Decay estimates}, we primarily invoke some useful tools of the Fourier analysis and the perturbation analysis in Section \ref{sec:Preliminaries}. With these tools, we construct the asymptotic expansions of the operator $E$ in \eqref{eq:operators E and A} in Section \ref{sec:Asymptotic expansions} in order to obtain the asymptotic expansions of the fundamental solution to the system \eqref{eq:the dissipative hyperbolic equation} to be able to prove the estimates in Section \ref{sec:Decay estimates}.

%%%%%%%%%%%%%%%NOTATIONS AND DEFINITIONS%%%%%%%%%%%
\subsection*{Notations and Definitions}
We introduce here the notations and definitions which will be used frequently through out this paper. See \cite{chemin11,bergh12} for more details.
\begin{definition}
Let $u$ be a function from $\mathbb R^d$ to a Banach space equipped with norm $|\cdot|$, we define the Lebesgue spaces $L^p(\mathbb R^d)$ for $1\le p\le \infty$ consisting of functions $u$ satisfying
\begin{equation*}
	\|u\|_{L^p}:=\Bigl(\int_{\mathbb R^d}|u(\mathbf x)|^p\,\dif \mathbf x\Bigr)^{1/p}<+\infty,\qquad 1\le p<\infty,
\end{equation*}
and satisfying
\begin{equation*}
	\|u\|_{L^\infty}:=\esssup_{\mathbb R^d} |u(\mathbf x)|<+\infty.
\end{equation*}
\end{definition}

Let $\alpha\in\mathbb N^d$ be the multi-index $\alpha:=(\alpha_1,\dots,\alpha_d)$ with $\alpha_j\in\mathbb N$. One denotes by
\begin{equation*}
	\partial^\alpha f:=\dfrac{\partial^{|\alpha|}f}{\partial x_1^{\alpha_1}\dots\partial x_d^{\alpha_d}},
\end{equation*}
where $|\alpha|:=\alpha_1+\dots+\alpha_d$, the partial derivatives of a smooth function $f$ on $\mathbb R^d$. Then, for smooths functions $f$ and $g$ on $\mathbb R^d$, we have the Leibniz rule
\begin{equation*}
	\partial^\alpha(fg)=\sum_{\nu\le \alpha}\begin{pmatrix} \alpha \\ \nu \end{pmatrix}\partial^\nu f\partial^{\alpha-\nu}g,
\end{equation*}
where
\begin{equation*}
	\begin{pmatrix} \alpha \\ \nu\end{pmatrix}:=\dfrac{\alpha!}{\nu!(\alpha-\nu)!}=\dfrac{\alpha_1!\dots \alpha_d!}{\nu_1!\dots \nu_d!(\alpha_1-\nu_1)!\dots (\alpha_d-\nu_d)!}
\end{equation*}
is the multi-index binomial coefficient, $\nu \le \alpha$ means that $\nu_j\le \alpha_j$ for all $j\in\{1,\dots,d\}$ and the difference $\alpha-\nu$ is defined by
\begin{equation*}
	\alpha-\nu:=(\alpha_1-\nu_1,\dots,\alpha_d-\nu_d).
\end{equation*}
\begin{definition}
The Schwartz space $\mathcal S(\mathbb R^d)$ is the set of smooth functions $u$ on $\mathbb R^d$ such that for any $k\in\mathbb N$, we have
\begin{equation*}
	\|u\|_{k,\mathcal S}:=\sup_{|\alpha|\le k,\, \mathbf x\in\mathbb R^d}(1+|\mathbf x|)^k|\partial^\alpha u(\mathbf x)|<+\infty.
\end{equation*}
\end{definition}
One denotes by $\mathcal S'(\mathbb R^d)$ the dual space of $\mathcal S(\mathbb R^d)$ and $u\in\mathcal S'(\mathbb R^d)$ is called a tempered distribution. For $u\in\mathcal S$, the Fourier transform $\hat u(\mathbf k)=\mathcal F(u(\mathbf x))$ is defined by
\begin{equation*}
	\hat u(\mathbf k):=\int_{\mathbb R^d}e^{-i\mathbf x\cdot \mathbf k}u(\mathbf x)\,\dif \mathbf x,
\end{equation*}
where $\mathbf x\cdot \mathbf k$ is the usual scalar product on $\mathbb R^d$, and the inverse Fourier transform of $\hat u$ also denoted by $u(\mathbf x)=\mathcal F^{-1}(\hat u(\mathbf k))$ is given by
\begin{equation*}
	u(\mathbf x):=\dfrac{1}{(2\pi)^d}\int_{\mathbb R^d}e^{i\mathbf x\cdot \mathbf k}\hat u(\mathbf k)\,\dif \mathbf k.
\end{equation*}
On the other hand, we can define the Fourier transform of tempered distributions $u\in\mathcal S'(\mathbb R^d)$ by the inner product $\langle \cdot,\cdot \rangle_{L^2} $ on $L^2(\mathbb R^d)$, namely
\begin{equation*}
	\langle \hat u(\mathbf k), \phi(\mathbf k) \rangle_{L^2}=\langle u (\mathbf x), \hat \phi(\mathbf x)\rangle_{L^2},\qquad \forall \phi\in\mathcal S(\mathbb R^d).
\end{equation*}
\begin{definition}
Let $s\in\mathbb R$, the Sobolev space $H^s(\mathbb R^d)$ consists of tempered distributions $u$ such that $\hat u\in L^{2}_{\textrm{loc}}(\mathbb R^d)$ and
\begin{equation*}
	\|u\|_{H^s}:=\Bigl(\int_{\mathbb R^d}(1+|\mathbf k|^2)^s|\hat u(\mathbf k)|^2\,\dif \mathbf k\Bigr)^{1/2}<+\infty.
\end{equation*}
\end{definition}
\begin{definition}
Let $\rho\in \mathcal S'(\mathbb R^d)$, $\rho$ is called a Fourier multiplier on $L^p(\mathbb R^d)$ for $1\le p\le \infty$ if the convolution $\mathcal F^{-1}(\rho(\mathbf k))*\phi\in L^p(\mathbb R^d)$ for all $\phi \in\mathcal S(\mathbb R^d)$ and if
\begin{equation*}
	\|\rho\|_{M_p}:=\sup_{\|\phi\|_{L^p}=1}\|\mathcal F^{-1}(\rho(\mathbf k))*\phi\|_{L^p}<+\infty.
\end{equation*}
The linear space of all such $\rho$ is denoted by $M_p(\mathbb R^d)$ equipped with norm $\|\cdot \|_{M_p}$.
\end{definition}

%%%%%%%%%%%%%%%PROOFS OF MAIN RESULTS%%%%%%%%%%%%
\section{Proofs and Examples of Theorem \ref{theo:LpLq estimates 1} and Theorem \ref{theo:LpLq estimates 2}}\label{sec:Proofs of the main theorems}
For $\mathbf k\in\mathbb R^d$, let $E=E(i\mathbf k)\in\mathbb R^{n\times n}$ be in \eqref{eq:operators E and A}. Let $\mathbf c\in\mathbb R^d$ and $\mathbf D\in\mathbb R^{d\times d}$ be in \eqref{eq:coefficients c and D 1}. Let $P_0^{(0)}\in\mathbb R^{n\times n}$ be in \eqref{eq:coefficients P0 and S0 1} and $\mathbf P_0^{(1)}\in (\mathbb R^{n\times n})^d$ be in \eqref{eq:coefficient P01 of the total projection P0 1}.

For $(\mathbf x,t)\in \mathbb R^d\times \mathbb R_+$, consider $\Gamma_t(\mathbf x):=\Gamma(\mathbf x,t)=\mathcal F^{-1}(e^{-E(i\mathbf k)t})\in\mathbb R^{n\times n}$, the kernel associated with the system \eqref{eq:the dissipative hyperbolic equation}, and $\tilde \Phi_t(\mathbf x):=\tilde \Phi(\mathbf x,t)=\mathcal F^{-1}(e^{-\mathbf c\cdot i\mathbf kt-\mathbf k\cdot \mathbf D\mathbf k t})\in\mathbb R$, the kernel associated with the system \eqref{eq:the parabolic equation}. Note that
\begin{equation}\label{eq:tilde Phi_t and Phi_t}
	\tilde \Phi_t*U_0(\mathbf x)=\Phi_t*u_0(\mathbf x),
\end{equation}
where
\begin{equation}\label{eq:Phi_t}
	\Phi_t(\mathbf x):=\Phi(\mathbf x,t)=\mathcal F^{-1}(e^{-\mathbf c\cdot i\mathbf kt-\mathbf k\cdot \mathbf D\mathbf k t}P_0^{(0)})\in\mathbb R^{n\times n}.
\end{equation}
Consider also the kernel $\tilde \Psi_t(\mathbf x):=\tilde \Psi(\mathbf x,t)=\mathcal F^{-1}(e^{-\mathbf k\cdot \mathbf D\mathbf kt})\in\mathbb R$ associated with the system \eqref{eq:the parabolic equation 2}. One has
\begin{equation}\label{eq:tilde Psi_t and Psi_t}
	\tilde \Psi_t*U_0(\mathbf x)=\Psi_t*u_0(\mathbf x),
\end{equation}
where
\begin{equation}\label{eq:Psi_t}
	\Psi_t(\mathbf x):=\Psi(\mathbf x,t)=\mathcal F^{-1}(e^{-\mathbf k\cdot \mathbf D\mathbf k t}(P_0^{(0)}+\mathbf P_0^{(1)}\cdot i\mathbf k))\in\mathbb R^{n\times n}.
\end{equation}

We are now able to give the proofs of Theorem \ref{theo:LpLq estimates 1} and Theorem \ref{theo:LpLq estimates 2} by using the estimates which will be proved later in Section \ref{sec:Decay estimates}.
\begin{proof}[Proof of Theorem \ref{theo:LpLq estimates 1}]
Let $u\in\mathbb R^n$ be the solution to \eqref{eq:the dissipative hyperbolic equation} with the initial data $u_0$ and $U\in\mathbb R^n$ be the solution to \eqref{eq:the parabolic equation} with the initial data $U_0$. One has
\begin{equation*}
	u(\mathbf x,t)=\Gamma_t*u_0(\mathbf x),\qquad U(\mathbf x,t)=\tilde \Phi_t*U_0(\mathbf x).
\end{equation*}
Moreover, by the relation \eqref{eq:tilde Phi_t and Phi_t}, one has
\begin{equation*}
	u(\mathbf x,t)-U(\mathbf x,t)=(\Gamma_t-\Phi_t)*u_0(\mathbf x),
\end{equation*}
where $\Phi_t$ is given by \eqref{eq:Phi_t}.

On the other hand, we decompose
\begin{equation*}
	u(\mathbf x,t)=u^{(1)}(\mathbf x,t)+u^{(2)}(\mathbf x,t),
\end{equation*}
where
\begin{equation*}
	u^{(1)}(\mathbf x,t)=\mathcal F^{-1}(e^{-E(i\mathbf k)t}P_0(i\mathbf k)\chi_1(\mathbf k))*u_0(\mathbf x)
\end{equation*}
and $u^{(2)}$ is the remainder, where $P_0$ is the eigenprojection associated with the eigenvalue of $E$ in \eqref{eq:operators E and A} converging to $0$ as $|\mathbf k|\to 0$ and $\chi_1$ is a cut-off function with support contained in the ball $B(\mathbf 0,\varepsilon)\subset \mathbb R^d$, valued in $[0,1]$, for small $\varepsilon>0$.

Therefore, by Proposition \ref{prop:low-frequency estimates}, Proposition \ref{prop:intermediate-frequency estimates 2} and Proposition \ref{prop:high-frequency estimates 2}, for $1\le q\le p\le \infty$, there is a constant $C>0$ such that we have
\begin{align*}
	\|u^{(1)}-U\|_{L^p}&\le \|\mathcal F^{-1}((\hat \Gamma_t(\mathbf k)P_0(i\mathbf k)-\hat \Phi_t(\mathbf k))\chi_1(\mathbf k))*u_0\|_{L^p}\nonumber\\
	&\hskip2cm+ \|\mathcal F^{-1}(\hat \Phi_t(\mathbf k)\chi_2(\mathbf k))*u_0\|_{L^p}+ \|\mathcal F^{-1}(\hat \Phi_t(\mathbf k)\chi_3(\mathbf k))*u_0\|_{L^p}\nonumber\\
	&\le Ct^{-\frac d2(\frac 1q-\frac 1p)-\frac 12}\|u_0\|_{L^p}, \qquad \forall t\ge 1,
\end{align*}
where $\chi_2:=1-\chi_1-\chi_3$ and $\chi_3$ is a cut-off function with support contained in $\{\mathbf k\in \mathbb R^d:|\mathbf k|>\rho\}$, valued in $[0,1]$, for large $\rho>0$.

Finally, by Proposition \ref{prop:low-frequency estimates}, Proposition \ref{prop:intermediate-frequency estimates} and Proposition \ref{prop:high-frequency estimates}, one also has
\begin{align*}
	\|u^{(2)}\|_{L^2}&\le \|\mathcal F^{-1}(\hat \Gamma_t(\mathbf k)(I-P_0(\mathbf k))\chi_1(\mathbf k))*u_0\|_{L^2}\nonumber\\
	&\hskip3cm+ \|\mathcal F^{-1}(\hat \Gamma_t(\mathbf k)(1-\chi_1(\mathbf k)))*u_0\|_{L^2}\nonumber\\
	&\le Ce^{-ct}\|u_0\|_{L^2},\qquad \forall t\ge 1,
\end{align*}
for some constants $c>0$ and $C>0$. The proof is done.
\end{proof}
%%%%%%%%%Example 1%%%%%%%%
\begin{example}
Consider the three-dimensional $3\times 3$ linear Goldstein--Kac system \eqref{eq:the dissipative hyperbolic equation} where
\begin{equation*}
	A^j=\begin{pmatrix}
	v_1^j & 0 & 0\\
	0 & v_2^j & 0\\
	0 & 0 & v_3^j
	\end{pmatrix},\qquad B=\begin{pmatrix}
		b+c & -c & -b\\
		-c & a+c & -a\\
		-b & -a & a+b
	\end{pmatrix},
\end{equation*}
where $v_i^j\in\mathbb R$ for $i,j\in\{1,2,3\}$ and $a,b,c> 0$, and the initial data is $u_0=(u_0^1,u_0^2,u_0^3)^T\in\mathbb R^3$. Let $\mathbf A_i=(v_i^1,v_i^2,v_i^3)$ for $i\in\{1,2,3\}$, we also consider the system \eqref{eq:the parabolic equation} where $\mathbf c=(\sum_{i=1}^3\mathbf A_i)/3$ and if $\mathbf c=\mathbf 0$, the matrix $\mathbf D$ is given by
\begin{equation*}
	\mathbf D=\dfrac{1}{3(ab+bc+ca)}\bigl(a\mathbf A_1\otimes \mathbf A_1+b\mathbf A_2\otimes \mathbf A_2+c\mathbf A_3\otimes \mathbf A_3\bigl).
\end{equation*}
Moreover, the initial data is chosen as
\begin{equation*}
	U_0=\dfrac{1}{3}(u_0^1+u_0^2+u_0^3,u_0^1+u_0^2+u_0^3,u_0^1+u_0^2+u_0^3)^T\in\mathbb R^3.
\end{equation*}
Theorem \ref{theo:LpLq estimates 1} then implies that the solution $u$ to the three-dimensional $3\times 3$ Goldstein--Kac system can be decomposed into $u=u^{(1)}+u^{(2)}$ such that the difference $u^{(1)}-U$ decays in $L^p(\mathbb R^d)$ at the rate $t^{-\frac 32(\frac 1q-\frac 1p)-\frac 12}$ with respect to $u_0$ in $L^q(\mathbb R^d)$ as $t\to +\infty$ for any $1\le q\le p\le \infty$, where $U$ is the solution to the above system \eqref{eq:the parabolic equation}. The formulas of $\mathbf c$ and $\mathbf D$ in fact coincide the formulas obtained by using the graph theory as in Example 3.3 p. 412 in \cite{mascia16}.
\end{example}
%%%%%%%%%%%%%Proof of Theorem 2%%%%%%%%%%%%%%%%
\vskip.25cm
We give the proof of Theorem \ref{theo:LpLq estimates 2}.
\begin{proof}[Proof of Theorem \ref{theo:LpLq estimates 2}]
The proof is similar to the proof of Theorem \ref{theo:LpLq estimates 1} where $\tilde \Phi_t$ and $\Phi_t$ are substituted by $\tilde \Psi_t$ and $\Psi_t$ respectively once considering $U$ to be the solution to \eqref{eq:the parabolic equation 2}. We finish the proof.
\end{proof}
%%%%%%%%%%%%%Example 2%%%%%%%%%%%%%%%%%%
\begin{example}
Consider the two-dimensional linearized isentropic Euler equations with damping
\begin{equation}\label{eq:invariant rotation}
\begin{cases}
	\partial_t\rho+\diver v=0,\\
	\partial_tv+\nabla_{\mathbf x}\rho=-v,
\end{cases}\quad (\mathbf x,t)\in \mathbb R^2\times \mathbb R_+,
\end{equation}
which can be written in the vectorial form
\begin{equation*}
	\partial_tu+A^1\partial_{x_1}u+A^2\partial_{x_2}u+Bu=0,
\end{equation*}
where $u=(\rho,v^1,v^2)^T\in\mathbb R^3$ with the initial data $u_0=(\rho_0,v^1_0,v^2_0)^T\in\mathbb R^3$ and
\begin{equation*}
	A^1=\begin{pmatrix}
	0 & 1 & 0\\
	1 & 0 & 0\\
	0 & 0 & 0
	\end{pmatrix},\qquad
	A^2=\begin{pmatrix}
	0 & 0 & 1\\
	0 & 0 & 0\\
	1 & 0 & 0
	\end{pmatrix},\qquad
	B=\begin{pmatrix}
	0 & 0 & 0\\
	0 & 1 & 0\\
	0 & 0 & 1
	\end{pmatrix}.
\end{equation*}
Moreover, the matrix $R$ satisfying the condition $\mathsf R$ and the matrix $S$ satisfying the condition $\mathsf S$ are given by
\begin{equation*}
	R(w_1,w_2)=\dfrac 12\begin{pmatrix}
	1 & 0 &1\\
	-w_1 & 2w_2 &w_1\\
	-w_2 & -2 w_1&w_2 
	\end{pmatrix},\qquad
	S=\begin{pmatrix}
	-1 & 0 & 0\\
	0 & 1 & 0\\
	0 & 0 & 1
	\end{pmatrix}.
\end{equation*}
Then, Theorem \ref{theo:LpLq estimates 2} implies that $u=u^{(1)}+u^{(2)}$, where $u^{(1)}$ has the asymptotic profile, which is the solution $U\in\mathbb R^3$ to the Cauchy problem
\begin{equation*}
	\begin{cases}\partial_t U -\Delta_{\mathbf x}U=0,&(\mathbf x,t)\in\mathbb R^2\times \mathbb R_+,\\
	U(\mathbf x,0)=U_0(\mathbf x),
	\end{cases}
\end{equation*}
where
\begin{equation*}
	U_0=(\rho_0-\partial_{x_1}v^1_0-\partial_{x_2}v^2_0,-\partial_{x_1}\rho_0,-\partial_{x_2}\rho_0)^T\in\mathbb R^3.
\end{equation*}
Moreover, $u^{(1)}-U$ decays in $L^p(\mathbb R^d)$ at the optimal rate $t^{-(\frac 1q-\frac 1p)-1}$ with respect to $u_0$ in $L^q(\mathbb R^d)$ as $t\to +\infty$ for any $1\le q\le p\le \infty$. This result is comparable with \cite{hosono04} since $\rho\in\mathbb R$ satisfying \eqref{eq:invariant rotation} also satisfies the linear damped wave equation
\begin{equation*}
	\begin{cases}
		\partial_{tt}^2\rho-\Delta_{\mathbf x}\rho+\partial_t\rho=0,& (\mathbf x,t)\in\mathbb R^2\times \mathbb R_+,\\
		\rho(\mathbf x,0)=\rho_0(\mathbf x),\\
		\partial_t\rho(\mathbf x,0)=-\partial_{x_1}v_0^1(\mathbf x)-\partial_{x_2}v^2_0(\mathbf x).
	\end{cases}
\end{equation*}
The proofs of Theorem 2.1 in \cite{hosono04} then implies that $\rho^{(1)}-\phi$ decays in $L^p(\mathbb R^d)$ at the rate $t^{-(\frac 1q-\frac 1p)-1}$ with respect to $(\rho_0,\partial_t\rho_0)$ in $L^q(\mathbb R^d)$ as $t\to +\infty$ for any $1\le q\le p\le \infty$, where $\phi$ is the solution to the heat equation
\begin{equation*}
	\begin{cases}
		\partial_{t}\phi-\Delta_{\mathbf x}\phi=0,& (\mathbf x,t)\in\mathbb R^2\times \mathbb R_+,\\
		\phi(\mathbf x,0)=\rho_0(\mathbf x)-\partial_{x_1}v^1_0(\mathbf x)-\partial_{x_2}v^2_0(\mathbf x).
	\end{cases}
\end{equation*}
\end{example}
%%%%%%%%%%%%%%Remark%%%%%%%%%%%%%%%%%%%
\begin{remark}[Proof of the case of finite speed of propagation]
In the case where $\Gamma_t$ has compact support contained in the wave cone $\{(\mathbf x,t)\in\mathbb R^d\times \mathbb R:|\mathbf x/t|\le C\}$ for some constant $C>0$, also by Proposition \ref{prop:low-frequency estimates} - Proposition \ref{prop:high-frequency estimates 2}, $u^{(1)}$ can be refined by
\begin{equation*}
	u^{(1)}(\mathbf x,t)=\mathcal F^{-1}(e^{-E(i\mathbf k)t}\chi_1(\mathbf k))*u_0(\mathbf x),
\end{equation*}
where $\chi_1$ is a cut-off function with support contained in the ball $B(\mathbf 0,\rho)\subset \mathbb R^d$, valued in $[0,1]$, for any $\rho>0$. The proof is then similar to the above proofs. Moreover, this property holds for the above two examples since they are in fact symmetric hyperbolic systems.
\end{remark}

%%%%%%%%%%%%%%%%PRELIMINARIES%%%%%%%%%%%%%

\section{Useful lemmas}\label{sec:Preliminaries}
This section is devoted to some useful facts of the Fourier analysis in \cite{chemin11,bergh12} and the perturbation analysis in \cite{kato}. They will be used in Section \ref{sec:Asymptotic expansions} and Section \ref{sec:Decay estimates}.
\subsection{Fourier analysis}
We introduce here the two well-known inequalities which are the Young inequality and the complex interpolation inequality. On the other hand, we also introduce a powerful Fourier multiplier estimate which is the estimate \eqref{eq:Carlson-Beurling Hs} given by Lemma \ref{lem:Carlson-Beurling}. The multiplier estimates are very helpful to study the $L^p$-$L^p$ estimate for $1\le p\le \infty$.
\begin{lemma}[Young's inequality]
For $1\le p,q,r\le \infty$ satisfying $1/p+1/q=1+1/r$ and any $f\in L^p(\mathbb R^d)$ and $g\in L^q(\mathbb R^d)$, one has $f*g\in L^r(\mathbb R^d)$ and the inequality
\begin{equation*}
	\|f*g\|_{L^r}\le \|f\|_{L^p}\|g\|_{L^q}.
\end{equation*}
\end{lemma}
\begin{proof}
See the proof of Lemma 1.4 p. 5 in \cite{chemin11}.
\end{proof}
\begin{lemma}[Complex interpolation inequality]
Consider a linear operator $T$ which continuously maps $L^{p_j}(\mathbb R^d)$ into $L^{q_j}(\mathbb R^d)$ for $1\le p_j,q_j\le \infty$ with $j\in\{0,1\}$. Let $\theta\in [0,1]$ be such that
\begin{equation*}
	\Bigl(\dfrac{1}{p_\theta},\dfrac{1}{q_\theta}\Bigr):=(1-\theta)\Bigl(\dfrac{1}{p_0},\dfrac{1}{q_0}\Bigr)+\theta\Bigl(\dfrac{1}{p_1},\dfrac{1}{q_1}\Bigr),
\end{equation*}
then $T$ continuously maps $L^{p_\theta}(\mathbb R^d)$ into $L^{q_\theta}(\mathbb R^d)$ and one has
\begin{equation*}
	\|T\|_{\mathcal L(L^{p_\theta};L^{q_\theta})}\le \|T\|_{\mathcal L(L^{p_0};L^{q_0})}\|T\|_{\mathcal L(L^{p_1};L^{q_1})}.
\end{equation*}
\end{lemma}
\begin{proof}
See the proof of Corollary 1.12 p. 12 in \cite{chemin11}.
\end{proof}
\begin{lemma}[Carlson--Beurling]\label{lem:Carlson-Beurling}
If $\rho\in H^s(\mathbb R^d)$ for $s>d/2$, $\rho\in M_p(\mathbb R^d)$ and for some constant $C>0$, one has the estimate
\begin{equation}\label{eq:Carlson-Beurling Hs}
	\|\rho\|_{M_p}\le C\|\rho\|_{L^2}^{1-\frac{d}{2s}}\Bigl(\sum_{|\alpha|=s}\|\partial^\alpha\rho\|_{L^2}\Bigr)^{\frac{d}{2s}}, \qquad  1\le p\le\infty.
\end{equation}
\end{lemma}
\begin{proof}
See the proof of Lemma 6.1.5 p.135 in \cite{bergh12}.
\end{proof}
\subsection{Perturbation analysis}
We consider the perturbation theory for linear operators in \cite{kato} that will be used for studying the asymptotic expansions of the fundamental solution to the system \eqref{eq:the dissipative hyperbolic equation}.
\vskip.25cm
Consider the operator $T(z)$ for $z\in\mathbb C$ having the form
\begin{equation}\label{eq:T(z)}
	T(z)=T^{(0)}+zT^{(1)}+z^2T^{(2)}+\dots, \qquad T^{(j)}\in\mathbb R^{n\times n}.
\end{equation}
Exceptional points of the analytic operator $T(z)$ in \eqref{eq:T(z)} for $z\in\mathbb C$ are defined to be points in where the the eigenvalues of $T(z)$ intersect. Nonetheless, they are of finite number in the plane. In the domain excluding these points, the operator $T(z)$ has $p$ holomorphic distinct eigenvalues with constant algebraic multiplicities. Moreover, the $p$ eigenprojections and the $p$ eigennilpotents associated with them are also holomorphic. In fact, the eigenvalues of $T(z)$ are solutions to the dispersion polynomial
\begin{equation*}
	\det(T(z)-\mu I)=0
\end{equation*}
with holomorphic coefficients. The eigenvalues of $T(z)$ are then branches of one or more than one analytic functions with algebraic singularities of at most order $n$. As a consequence, the number of eigenvalues of $T(z)$ is a constant except for a number of points which is finite in each compact set of the plane. The exceptional points can be either regular points of the analytic functions or branch-points of some eigenvalues of $T(z)$. In the former case, the eigenprojections and the eigennilpotents associated with the eigenvalues are bounded while in the latter case, they have poles at the exceptional points even if the eigenvalues are continuous there (see \cite{kato}).

We study the behavior of the eigenvalues of $T(z)$ and the associated eigenprojections and eigennilpotents near an exceptional point. Without loss of generality, we assume that the exceptional point is the point $0\in\mathbb C$. Let $\lambda^{(0)}$ be an eigenvalue of $T^{(0)}$ with algebraic multiplicity $m\ge1$ and let $P^{(0)}$ and $N^{(0)}$ be the associated eigenprojection and eigennilpotent. One has
\begin{equation*}
	T^{(0)}P^{(0)}=P^{(0)}T^{(0)}=P^{(0)}T^{(0)}P^{(0)}=\lambda^{(0)}P^{(0)}+N^{(0)},
\end{equation*}
\begin{equation*}
	\dim P^{(0)}=m,\qquad (N^{(0)})^m=O,\qquad P^{(0)}N^{(0)}=N^{(0)}P^{(0)}.
\end{equation*}
The eigenvalue $\lambda^{(0)}$ is in general split into several eigenvalues of $T(z)$ for small $z\ne 0$. The set of these eigenvalues is called the $\lambda^{(0)}$-group. The total projection of this group, denoted by $P(z)$, is holomorphic at $z=0$ and is approximated by
\begin{equation}\label{eq:Total projection}
	P(z)=P^{(0)}+zP^{(1)}+z^2P^{(2)}+\mathcal O(|z|^3),
\end{equation}
where $P^{(j)}$ can be computed in terms of the coefficients $T^{(j)}$ in \eqref{eq:T(z)} and the coefficients $N^{(0)}$, $P^{(0)}$ and $Q^{(0)}$ given respectively by $N^{(0)}=(T^{(0)}-\lambda^{(0)}I)P^{(0)}$ and 
\begin{equation}\label{eq:P0 and S0}
	P^{(0)}=-\dfrac{1}{2\pi i}\int_{\Gamma}(T^{(0)}-\mu I)^{-1}\,\dif\mu, \qquad Q^{(0)}=\dfrac{1}{2\pi i}\int_{\Gamma}\mu^{-1}(T^{(0)}-\mu I)^{-1}\,\dif \mu,
\end{equation}
where $\Gamma$, in the resolvent set of $T^{(0)}$, is an oriented closed curve enclosing $\lambda^{(0)}$ except for the other eigenvalues of $T^{(0)}$. In fact, from \cite{kato} (eq. (2.13) p. 76), one has
\begin{equation}\label{eq:coefficients P1 and P2}
	P^{(1)}=\sum_{i+j=1}X^{(i)}T^{(1)}X^{(j)}, \quad P^{(2)}=\sum_{i+j=1}X^{(i)}T^{(2)}X^{(j)}-\sum_{i+j+h=2}X^{(i)}T^{(1)}X^{(j)}T^{(1)}X^{(h)},
\end{equation}
where
\begin{equation}\label{eq:coefficients Xi}
	X^{(0)}=P^{(0)},\qquad X^{(i)}=(Q^{(0)})^i,\qquad X^{(-i)}=-(N^{(0)})^i,\qquad  \forall i\ge 1.
\end{equation}
Moreover, the subspace $\ran P(z):=P(z)\mathbb C^n$ is $m$-dimensional and invariant under $T(z)$. The $\lambda^{(0)}$-group eigenvalues of $T(z)$ are identical with all the eigenvalues of $T(z)$ in $\ran P(z)$. In order to determine the $\lambda^{(0)}$-group eigenvalues, therefore, we have only sole an eigenvalue problem in the subspace $\ran P(z)$, which is in general smaller than the whole space $\mathbb C^n$.

The eigenvalue problem for $T(z)$ in $\ran P(z)$ is equivalent to the eigenvalue problem for
\begin{equation}\label{eq:Tr(z)}
	T_r(z)=T(z)P(z)=P(z)T(z)=P(z)T(z)P(z).
\end{equation}
Thus, the $\lambda^{(0)}$-group eigenvalues of $T(z)$ are exactly those eigenvalues of $T_r(z)$ which are different from $0$, provided $|\lambda^{(0)}|$ is large enough to ensure that these eigenvalues do not vanish for the small $z$ under consideration. The last condition does not restrict the generality, for $T^{(0)}$ could be replaced by $T^{(0)}+\alpha$ with a suitable scalar $\alpha$ without changing the nature of the problem (see \cite{kato}).

We also have the following result in \cite{kato}.
\begin{lemma}[A simple case]\label{prop:a simple case}
If $T(z)=T^{(0)}+zT^{(1)}$ and $\lambda^{(0)}$ is a simple eigenvalue of $T^{(0)}$, the eigenvalue $\lambda(z)$ of $T(z)$ converging to $\lambda^{(0)}$ as $|z|\to 0$ and its associated eigenprojection $P(z)$ are holomorphic at $z=0$.
Moreover, for small $z\ne 0$, $P(z)$ is approximated by \eqref{eq:Total projection} with the coefficients $P^{(j)}$ for $j=0,1,2,\dots$ and $\lambda(z)$ is approximated by
\begin{equation}\label{eq:simple eigenvalue lambda}
	\lambda(z)=\lambda^{(0)}+z\lambda^{(1)}+z^2\lambda^{(2)}+\mathcal O(|z|^3),
\end{equation}
where
\begin{equation}\label{eq:coefficient lambdaj}
	\lambda^{(j)}=\dfrac{1}{j}\trace(T^{(1)}P^{(j-1)}),\qquad j=1,2,3,\dots
\end{equation}
On the other hand, the eigennilpotent associated with $\lambda(z)$ which is $N(z)=\bigl(T(z)-\lambda(z)I\bigr)P(z)$ vanishes identically.
\end{lemma}
\begin{proof}
For any eigenvalue $\lambda^{(0)}$ of $T^{(0)}$ with algebraic multiplicity $m\ge 1$, one considers the weighted mean of the $\lambda^{(0)}$-group defined by
\begin{equation*}
	\hat \lambda(z):=\dfrac{1}{m}\trace\bigl(T(z)P(z)\bigr)=\lambda^{(0)}+\dfrac{1}{m}\trace\bigl((T(z)-\lambda^{(0)}I)P(z)\bigr),
\end{equation*}
where $P(z)$ is the total projection associated with the $\lambda^{(0)}$-group.

We study the asymptotic expansions of $\hat \lambda(z)$ and $P(z)$ for small $z\ne 0$. The expansion of $P(z)$ is given by \eqref{eq:Total projection} and following \cite{kato} (eq. (2.8) p. 76), the coefficient $P^{(j)}$ in \eqref{eq:Total projection} satisfies
\begin{equation}\label{eq:coefficient Pj}
	P^{(j)}=-\dfrac{1}{2\pi i}\sum_{\substack{ \nu_1+\dots +\nu_p=j\\ \nu_i\ge 1,\, i=1,\dots,p}}(-1)^{p}\int _{\Gamma}R^{(0)}(\zeta)T^{(\nu_1)}R^{(0)}(\zeta)T^{(\nu_2)}\dots T^{(\nu_p)}R^{(0)}(\zeta)\,\dif\zeta,
\end{equation}
where $T^{(\nu_i)}$ for $i\in\{1,\dots,p\}$ are the coefficients in \eqref{eq:T(z)}, $R^{(0)}(\zeta):=(T^{(0)}-\zeta I)^{-1}$ is the resolvent of $T^{(0)}$ and $\Gamma$ is a small positively-oriented circle around $\lambda^{(0)}$. On the other hand, following \cite{kato} (eq. (2.21) p.78 and eq. (2.30) p.79), the weighted mean $\hat \lambda(z)$ of the $\lambda^{(0)}$-group is approximated by
\begin{equation}\label{eq:weighted mean hatlambda}
	\hat\lambda(z)=\lambda^{(0)}+z\hat\lambda^{(1)}+z^2\hat\lambda^{(2)}+\mathcal O(|z|^3),
\end{equation}
where the coefficient $\hat \lambda^{(j)}$ is given by
\begin{equation}\label{eq:coefficient hatlambdaj}
	\hat \lambda^{(j)}=\dfrac{1}{2\pi i m}\trace\Bigl(\sum_{\substack{ \nu_1+\dots +\nu_p=j\\ \nu_i\ge 1,\, i=1,\dots,p}}\dfrac{(-1)^{p}}{p}\int _{\Gamma}T^{(\nu_1)}R^{(0)}(\zeta)\dots R^{(0)}(\zeta)T^{(\nu_p)}R^{(0)}(\zeta)\,\dif \zeta\Bigr),
\end{equation}
where the relative coefficients are introduced before.

In the case where $T(z)=T^{(0)}+zT^{(1)}$, one has $T^{(j)}=O$, the null matrix, for $j\ge 2$. Furthermore, since $\nu_i$ in \eqref{eq:coefficient Pj} and \eqref{eq:coefficient hatlambdaj} satisfying $\nu_i\ge1 $, it implies that $\nu_i=1$ for all $i$. Hence, we obtain from \eqref{eq:coefficient Pj} and \eqref{eq:coefficient hatlambdaj} that
\begin{equation}\label{eq:coefficient hatlambdaj reduced}
	\hat\lambda^{(j)}=\dfrac{1}{mj}\trace \Bigl(T^{(1)} \Bigl(\dfrac{(-1)^{j}}{2\pi i}\int_{\Gamma}R^{(0)}(\zeta)T^{(1)}\dots T^{(1)}R^{(0)}(\zeta)\,\dif\zeta\Bigr)\Bigr)=\dfrac{1}{mj}\trace\bigl(T^{(1)}P^{(j-1)}\bigr).
\end{equation}

If $\lambda^{(0)}$ is a simple eigenvalue, one has $m=1$ and $\lambda^{(0)}$ is not split into many eigenvalues of $T(z)$. Thus, the $\lambda^{(0)}$-group contains only one single eigenvalue $\lambda(z)$ of $T(z)$ converging to $\lambda^{(0)}$ as $|z|\to 0$. Hence, $\lambda(z)=\hat \lambda(z)$ and the eigenprojection associated with $\lambda(z)$ is exactly the total projection $P(z)$ of the $\lambda^{(0)}$-group. Therefore, one obtains the expansion \eqref{eq:simple eigenvalue lambda} from \eqref{eq:weighted mean hatlambda} and one obtains the formula \eqref{eq:coefficient lambdaj} from \eqref{eq:coefficient hatlambdaj reduced}, where $m=1$. The eigenilpotent $N(z)$ associated with $\lambda(z)$ is obviously null since $\lambda(z)$ is simple. The proof is done.
\end{proof}
 
Moreover, one obtains the following result from Lemma \ref{prop:a simple case}. 
\begin{corollary}[Symmetry]\label{prop:symmetry}
Under the same assumptions of Lemma \ref{prop:a simple case}, if in addition, there is an invertible matrix $S\in \mathbb R^{n\times n}$ such that $ST^{(1)}=-T^{(1)}S$ and $ST^{(0)}=T^{(0)}S$, then $\lambda^{(j)}=0$ for all $j$ odd, where $\lambda^{(j)}$ for $j=1,2,\dots$ is the $j$-th coefficient in the formulas \eqref{eq:simple eigenvalue lambda} and \eqref{eq:coefficient lambdaj}.
\end{corollary}
\begin{proof}
Recall $T(z)=T^{(0)}+zT^{(1)}$, one can study the eigenvalue problem for $T(z)$ by considering the operator
\begin{equation}\label{eq:TS(z)}
	T_S(z):=ST(z)S^{-1}=ST^{(0)}S^{-1}+zST^{(1)}S^{-1}=T^{(0)}-zT^{(1)}=T_S^{(0)}+zT_S^{(1)},
\end{equation}
where $T_S^{(0)}:=T^{(0)}$ and $T_S^{(1)}:=-T^{(1)}$. Thus, Lemma \ref{prop:a simple case} is applied to $T_S(z)$ since $\lambda^{(0)}$ is also a simple eigenvalue of $T_S^{(0)}$. It implies that the eigenvalue $\lambda_S(z)$ of $T_S(z)$ converging to $\lambda^{(0)}$ as $|z|\to 0$ and the associated eigenprojection $P_S(z)$ are holomorphic at $z=0$. Moreover, for small $z\ne 0$, the expansion of $P_S(z)$ is given by \eqref{eq:Total projection} with coefficients denoted by $P_S^{(j)}$ for $j=0,1,2,\dots$ and $\lambda_S(z)$ is approximated by
\begin{equation*}
	\lambda_S(z)=\lambda^{(0)}+z\lambda_S^{(1)}+z^2\lambda_S^{(2)}+\mathcal O(|z|^3),
\end{equation*}
where
\begin{equation}\label{eq:coefficient lambdaSj}
	\lambda_S^{(j)}=\dfrac{1}{j}\trace(T_S^{(1)}P_S^{(j-1)}),\qquad j=1,2,3,\dots
\end{equation}
On the other hand, the eigennilpotent $N_S(z)$ associated with $\lambda_S(z)$ vanishes identically.

Consider the total projection $P_S(z)$ associated with the $\lambda^{(0)}$-group of $T_S(z)$ in \eqref{eq:Total projection} with the coefficients $P_S^{(j)}$. We also consider the formula \eqref{eq:coefficient Pj} of $P_S^{(j)}$, namely
\begin{equation*}
	P_S^{(j)}=-\dfrac{1}{2\pi i}\sum_{\substack{ \nu_1+\dots +\nu_p=j\\ \nu_i\ge 1,\, i=1,\dots,p}}(-1)^{p}\int _{\Gamma}R_S^{(0)}(\zeta)T_S^{(\nu_1)}R_S^{(0)}(\zeta)T_S^{(\nu_2)}\dots T_S^{(\nu_p)}R_S^{(0)}(\zeta)\,\dif\zeta,
\end{equation*}
where $T_S^{(\nu_i)}$ for $i\in\{1,\dots,p\}$ are the coefficients in the expansion \eqref{eq:T(z)} of $T_S(z)$, $R_S^{(0)}(\zeta):=(T_S^{(0)}-\zeta I)^{-1}$ is the resolvent of $T_S^{(0)}$ and $\Gamma$ is a small positively-oriented circle around $\lambda^{(0)}$. Then, since $T_S^{(\nu_i)}=O$ for all $\nu_i\ge 2$ and since $\nu_i\ge 1$ for all $i$, one has
\begin{equation*}
	P_S^{(j)}=-\dfrac{1}{2\pi i}(-1)^{j}\int _{\Gamma}R_S^{(0)}(\zeta)T_S^{(1)}R_S^{(0)}(\zeta)T_S^{(1)}\dots T_S^{(1)}R_S^{(0)}(\zeta)\,\dif\zeta.
\end{equation*}
Since $T_S^{(0)}=T^{(0)}$ and $T_S^{(1)}=-T^{(1)}$, it follows that for all $j$, one has
\begin{equation}\label{eq:coefficient PSj}
	P_S^{(j)}=\begin{cases}
		P^{(j)}&\textrm{if }j \textrm{ is even},\\
		-P^{(j)}&\textrm{if }j\textrm{ is odd},
	\end{cases}
\end{equation}
where $P^{(j)}$ is the $j$-th coefficient in the expansion of the total projection $P(z)$ associated with the $\lambda^{(0)}$-group of $T(z)=T^{(0)}+zT^{(1)}$.

Hence, from \eqref{eq:coefficient lambdaj}, \eqref{eq:coefficient lambdaSj} and \eqref{eq:coefficient PSj}, we have
\begin{equation}\label{eq:coefficient lambdaSj 2}
	\lambda_S^{(j)}=\begin{cases}
		\lambda^{(j)}&\textrm{if }j \textrm{ is even},\\
		-\lambda^{(j)}&\textrm{if }j\textrm{ is odd},
	\end{cases}
\end{equation}
where $\lambda_j$ is the $j$-th coefficient in the expansion of the eigenvalue $\lambda(z)$ of $T(z)=T^{(0)}+zT^{(1)}$ converging to $\lambda^{(0)}$ as $|z|\to 0$.

Finally, since $\lambda_S(z)\equiv \lambda(z)$ due to \eqref{eq:TS(z)} and the fact that they are single eigenvalues, we deduce from \eqref{eq:coefficient lambdaSj 2} that $\lambda^{(j)}=-\lambda^{(j)}=0$ for all $j$ odd. We finish the proof.
\end{proof}
Let $\sigma(T,\mathcal D)$ be the spectrum of $T$ considered in the domain $\mathcal D$, we finish this section by introducing the reduction method in \cite{kato} which can be applied for the semi-simple-eigenvalue case.
\begin{lemma}[Reduction process]\label{prop:reduction process}
Let $T(z)$ be in \eqref{eq:T(z)} with the coefficients $T^{(i)}$ for $i=0,1,2,\dots$ and let $\lambda^{(0)}$ be a semi-simple eigenvalue of $T^{(0)}$. Let $P(z)$ in \eqref{eq:Total projection} with the coefficients $P^{(i)}$ for $i=0,1,2,\dots$ be the total projection of the $\lambda^{(0)}$-group. The following holds for small $z\ne 0$
\begin{equation}\label{eq:the division of T(z)P(z)}
	T(z)P(z)=\sum_{j=1}^p(\lambda^{(0)}I+zT_j(z))P_j(z),
\end{equation}
where $T_j(z)$ commutes with $P_j(z)$ and $P_j(z)$ satisfies
\begin{equation}\label{eq:subprojections}
	P_j(z)P_{j'}(z)=\delta_{jj'}P_j(z),\qquad \sum_{j=1}^pP_j(z)=P(z).
\end{equation}
The expansions of $T_j(z)$ and $P_j(z)$ are
\begin{equation}\label{eq:Tj(z)}
	T_j(z)=\lambda^{(0)}_jI+N^{(0)}_j+\mathcal O(|z|)
\end{equation}
and
\begin{equation}\label{eq:subprojection Pj(z)}
	P_j(z)=P_j^{(0)}+\mathcal O(|z|),
\end{equation}
where $\lambda^{(0)}_j\in\sigma(P^{(0)}T^{(1)}P^{(0)},\ker\,(T^{(0)}-\lambda^{(0)}I))$ with the associated eigenprojection $P_j^{(0)}$ and eigennilpotent $N_j^{(0)}$ for $j\in\{1,\dots,p\}$ and $p$ is the cardinality of $\sigma(P^{(0)}T^{(1)}P^{(0)},\ker\,(T^{(0)}-\lambda^{(0)}I))$.
\end{lemma}
\begin{proof}
Recall $T(z)$ and the coefficients $T^{(j)}$ in \eqref{eq:T(z)}. Recall the expansion of the total projection $P(z)$ of the $\lambda^{(0)}$-group of $T(z)$, generated by the eigenvalue $\lambda^{(0)}$ of $T^{(0)}$, and the coefficients $P^{(j)}$ in \eqref{eq:Total projection}. If $\lambda^{(0)}$ is semi-simple, one obtains from \eqref{eq:Tr(z)} that $(T(z)-\lambda^{(0)}I)P(z)=z\tilde T(z)$, where
\begin{equation}\label{eq:tildeT(z)}
	\tilde T(z)=\tilde T^{(0)}+z\tilde T^{(1)}+\mathcal O(|z|^2),
\end{equation}
where $\tilde T^{(0)}:=P^{(0)}T^{(1)}P^{(0)}$ and $\tilde T^{(1)}:=P^{(1)}T^{(0)}P^{(1)}+P^{(1)}T^{(1)}P^{(0)}+P^{(0)}T^{(1)}P^{(1)}$. Thus, the eigenvalues of $\tilde T(z)$ in $\ran P(z)$ are considered and in general, they converge to the eigenvalues of $\tilde T^{(0)}$ in $\ran P^{(0)}=\ker\,(T^{(0)}-\lambda^{(0)}I)$ as $|z|\to 0$ (see Theorem 2.3 p. 82 in \cite{kato}). One denotes the distinct eigenvalues of $\tilde T^{(0)}$ in $\ker(T^{(0)}-\lambda^{(0)}I)$ by $\lambda_j^{(0)}$ for $j\in\{1,\dots,p\}$. Then, $\lambda_j^{(0)}$ generates the $\lambda_j^{(0)}$-group of $\tilde T(z)$ similarly to the $\lambda^{(0)}$-group of $T(z)$ generated by the eigenvalue $\lambda^{(0)}$ of $T^{(0)}$. Moreover, the total projection $P_j(z)$ of the $\lambda_j^{(0)}$-group commutes with $\tilde T(z)$, satisfies \eqref{eq:subprojections} and is approximated by \eqref{eq:subprojection Pj(z)}.

Applying again \eqref{eq:Tr(z)} where $T(z)$ is substituted by $\tilde T(z)$ and $P(z)$ is substituted by $P_j(z)$, it follows from \eqref{eq:tildeT(z)} and \eqref{eq:subprojection Pj(z)} that
\begin{equation}\label{eq:T(z)Pj(z)}
	\tilde T(z) P_j(z)=\lambda_j^{(0)}I+N_j^{(0)}+\mathcal O(|z|),
\end{equation}
where $N_j^{(0)}$ is the eigennilpotent associated with $\lambda_j^{(0)}$. Let $T_j(z):=\tilde T(z)P_j(z)$ and using \eqref{eq:subprojections}, \eqref{eq:T(z)Pj(z)} and the fact that $T(z)P(z)=z\tilde T(z)$, one obtains \eqref{eq:the division of T(z)P(z)} and \eqref{eq:Tj(z)}. We finish the proof.
\end{proof}

%%%%%%%%%%%%%ASYMPTOTIC EXPANSIONS%%%%%%%%%%%%%
\section{Preliminaries to Section \ref{sec:Decay estimates}}\label{sec:Asymptotic expansions}
In this section, we study the asymptotic expansions of $E(i\mathbf k)=B+A(i\mathbf k)$ in \eqref{eq:operators E and A} for $\mathbf k\in\mathbb R^d$, which will be used in Section \ref{sec:Decay estimates}. One has
\begin{equation}\label{eq:radius-angular form of E(z)}
	E(i\mathbf k)=E(\zeta,\mathbf w):=B+i\zeta A(\mathbf w),
\end{equation}
where $\zeta:=|\mathbf k |\in [0,+\infty)$ and $\mathbf w:=\mathbf k/|\mathbf k|\in \mathbb S^{d-1}$. Moreover, since $\mathbb S^{d-1}$ is compact, $\zeta=0$ is an isolated exceptional point of $E(\zeta,\mathbf w)$ uniformly for $\mathbf w\in\mathbb S^{d-1}$ while there is a finite number of exceptional curves of $E(\zeta,\mathbf w)$ for $0<\zeta<+\infty$. The exceptional point $\zeta=+\infty$ is not a uniform exceptional point for $\mathbf w\in\mathbb S^{d-1}$ in general (see \cite{bianchini07,kato}). Nonetheless, we can approximate $E(\zeta,\mathbf w)$ near $\zeta=+\infty$ by subtracting a suitable Lebesgue measure zero set taken advantage of the conditions $\mathsf A$ and $\mathsf R$. In this paper, we are only interested in the asymptotic expansions of $E(\zeta,\mathbf w)$ near $\zeta=0$ and $\zeta=+\infty$. As a consequence of Lemma \ref{prop:a simple case} and Lemma \ref{prop:reduction process}, we obtain the followings.

\begin{proposition}[Low-frequency approximation]\label{prop:low-frequency}
If the assumptions $\mathsf B$ and $\mathsf D$ hold, then for small $\mathbf k\in\mathbb R^d$, $E(i\mathbf k)$ is approximated by
\begin{equation}\label{eq:E(z) for low frequency}
	E(i\mathbf k)=\lambda_0(i\mathbf k)P_0(i\mathbf k)+\sum_{j=1}^sE_j(i\mathbf k)P_j(i\mathbf k),
\end{equation}
where
\begin{equation}\label{eq:the 0-group}
	\lambda_0(i\mathbf k)=\mathbf c \cdot i\mathbf k+\mathbf k \cdot \mathbf D \mathbf k+\mathcal O(|\mathbf k|^3),
\end{equation}
where $\mathbf c=(c_h)\in \mathbb R^d$ and $\mathbf D=(D_{h\ell}) \in \mathbb R^{d\times d}$ is positive definite with scalar entries
\begin{equation}\label{eq:coefficients c and D}
	c_h=\trace\bigl(A^h P_0^{(0)}\bigr),\qquad D_{h\ell}=\dfrac{1}{2}\trace\bigl(A^{h}P_0^{(0)}A^\ell Q_0^{(0)}+A^h Q_0^{(0)}A^\ell P_0^{(0)}\bigr),
\end{equation}
and
\begin{equation}\label{eq:the total projection P0 of the 0-group}
	P_0(i\mathbf k)=P_0^{(0)}+\mathbf P_0^{(1)}\cdot i\mathbf k+\mathcal O(|\mathbf k|^2),
\end{equation}
where $\mathbf P_0^{(1)}=(P_{0h}^{(1)})\in (\mathbb R^{n\times n})^d$ with matrix entries
\begin{equation}\label{eq:coefficient P01 of the total projection P0}
	P_{0h}^{(1)}=-P_0^{(0)}A^hQ_0^{(0)}-Q_0^{(0)}A^hP_0^{(0)},
\end{equation}
and $E_j(i\mathbf k)$ commutes with $P_j(i\mathbf k)$ and one has
\begin{equation}\label{eq:Ej(z) for low frequency}
	E_j(i\mathbf k)=\lambda_j^{(0)}I+N_j^{(0)}+\mathcal O(|\mathbf k|),
\end{equation}
and
\begin{equation}\label{eq:the total projection Pj of the other groups}
	P_j(i\mathbf k)=P_j^{(0)}+\mathcal O(|\mathbf k|),
\end{equation}
where $\lambda_j^{(0)}$ with $\re \lambda_j^{(0)}>0$ is the $j$-th nonzero eigenvalue of $B$ with the associated eigenprojection $P_j^{(0)}$ and eigennilpotent $N_j^{(0)}$ for $j\in\{1,\dots,s\}$ and $s$ is the number of the distinct nonzero eigenvalues of $B$.
\vskip.15cm
Moreover, if the condition $\mathsf S$ holds in addition, we have
\begin{equation}\label{eq:the 0-group 2}
	\lambda_0(i\mathbf k)=\mathbf k \cdot \mathbf D \mathbf k+\mathcal O(|\mathbf k|^4).
\end{equation}
\end{proposition}
\begin{proof}
We primarily consider the $0$-group of $E(\zeta,\mathbf w)$ in \eqref{eq:radius-angular form of E(z)} for small $\zeta>0$ and $\mathbf w\in\mathbb S^{d-1}$. Recall the spectrum $\sigma(B)$ of $B$. Since $0\in\sigma(B)$ is simple if the assumption $\mathsf B$ holds, the eigennilpotent $N_0^{(0)}$ associated with $0\in\sigma(B)$ is a null matrix and one obtains from \eqref{eq:Total projection}, \eqref{eq:coefficients P1 and P2} and \eqref{eq:coefficients Xi} that the total projection $P_0(\zeta,\mathbf w)$ of the $0$-group is approximated by
\begin{equation}\label{eq:radius-angular P0}
	P_0(\zeta,\mathbf w)=P_0^{(0)}+i\zeta P_0^{(1)}(\mathbf w)+\mathcal O(\zeta^2),
\end{equation}
where $P_0^{(0)}$ is the eigenprojection associated with $0\in\sigma(B)$ and
\begin{equation}\label{eq:radius-angular P01}
	P_0^{(1)}(\mathbf w)=-P_0^{(0)}A(\mathbf w)Q_0^{(0)}-Q_0^{(0)}A(\mathbf w)P_0^{(0)}=-\sum_{h=1}^d(P_0^{(0)}A^hQ_0^{(0)}+Q_0^{(0)}A^hP_0^{(0)})w_h.
\end{equation}
On the other hand, by \eqref{eq:simple eigenvalue lambda} and \eqref{eq:coefficient lambdaj} in Lemma \ref{prop:a simple case}, the $0$-group of $E(\zeta,\mathbf w)$ consists of one single eigenvalue $\lambda_0(\zeta,\mathbf w)$ approximated by
\begin{equation}\label{eq:radius-angular lambda0}
	\lambda_0(\zeta,\mathbf w)=i\zeta \lambda_0^{(1)}(\mathbf w)-\zeta^2\lambda_0^{(2)}(\mathbf w)+\mathcal O(\zeta^3),
\end{equation}
where
\begin{equation}\label{eq:radius-angular lambda01}
	\lambda_0^{(1)}(\mathbf w)=\trace(A(\mathbf w)P_0^{(0)})=\sum_{h=1}^d\trace(A^hP_0^{(0)})w_h
\end{equation}
and
\begin{equation}\label{eq:radius-angular lambda02}
	\lambda_0^{(2)}(\mathbf w)=\dfrac{1}{2}\trace(A(\mathbf w)P_0^{(1)}(\mathbf w))=-\dfrac{1}{2}\sum_{h,\ell=1}^d\trace (A^hP_0^{(0)}A^\ell Q_0^{(0)}+A^hQ_0^{(0)}A^\ell P_0^{(0)})w_h w_\ell.
\end{equation}

We consider the other groups of $E(\zeta,\mathbf w)$ for small $\zeta>0$. Let $\lambda_j^{(0)}\in\sigma(B)\backslash\{0\}$ be the $j$-th nonzero eigenvalue of $B$ for $j\in\{1,\dots,s\}$, one deduces directly from \eqref{eq:Total projection} that the approximation of the total projection $P_j(\zeta,\mathbf w)$ of the $\lambda_j^{(0)}$-group is given by
\begin{equation}\label{eq:radius-angular Pj}
	P_j(\zeta,\mathbf w)=P_j^{(0)}+\mathcal O(\zeta),
\end{equation}
where $P_j^{(0)}$ is the eigenprojection associated with $\lambda_j^{(0)}\in\sigma(B)\backslash \{0\}$. Moreover, due to the discussion above \eqref{eq:Tr(z)}, the study of the $\lambda_j^{(0)}$-group of $E(\zeta,\mathbf w)$ is equivalent to the study of the eigenvalues of $E_j(\zeta,\mathbf w)=E(\zeta,\mathbf w)P_j(\zeta,\mathbf w)$ in $\ran P_j(\zeta,\mathbf w)$. Furthermore, one has
\begin{equation}\label{eq:radius-angular Ej}
	E_j(\zeta,\mathbf w)=(B+i\zeta A(\mathbf w))(P_j^{(0)}+\mathcal O(\zeta))=BP_j^{(0)}+\mathcal O(\zeta)=\lambda_j^{(0)}I+N_j^{(0)}+\mathcal O(\zeta),
\end{equation}
where $N_j^{(0)}=(B-\lambda_j^{(0)}I)P_j^{(0)}$ is the eigennilpotent associated with $\lambda_j^{(0)}\in\sigma(B)\backslash \{0\}$. On the other hand, by definition, one also has $E_j(\zeta,\mathbf w)$ commutes with $P_j(\zeta,\mathbf w)$.

Finally, since $\sum_{j=0}^sP_j(\zeta,\mathbf w)=I$, the identity matrix, one has
\begin{equation}\label{eq:radius-angular E}
	E(\zeta,\mathbf w)=\sum_{j=0}^sE(\zeta,\mathbf w)P_j(\zeta,\mathbf w)=\lambda_0(\zeta,\mathbf w)P_0(\zeta,\mathbf w)+\sum_{j=1}^sE_j(\zeta,\mathbf w)P_j(\zeta, \mathbf w).
\end{equation}
We thus obtain \eqref{eq:E(z) for low frequency} - \eqref{eq:the total projection Pj of the other groups} once considering \eqref{eq:radius-angular P0} - \eqref{eq:radius-angular E} in the coordinates $\mathbf k\in\mathbb R^d$ except for the fact that the matrix $\mathbf D$ in \eqref{eq:coefficients c and D} is positive definite.

We now prove that $\mathbf D$ is positive definite. Consider the eigenvalue $\lambda_0(i\mathbf k)$ in \eqref{eq:the 0-group} of $E(i\mathbf k)$ for $\mathbf k\in\mathbb R^d$ with the coefficients $\mathbf c\in \mathbb R^d$ and $\mathbf D\in\mathbb R^{d\times d}$ given by \eqref{eq:coefficients c and D}. If the assumption $\mathsf D$ holds, then since $\mathbf c\cdot \mathbf k\in \mathbb R$, there is a constant $\theta >0$ such that for small $\mathbf k\ne \mathbf 0\in\mathbb R^d$, one has
\begin{equation*}
	\theta |\mathbf k|^2/(1+|\mathbf k|^2)\le \re \lambda_0(i\mathbf k)\le \re (\mathbf k\cdot \mathbf D\mathbf k)+C|\mathbf k|^3.
\end{equation*}
As $|\mathbf k|\to 0$, one has $\re (\mathbf w\cdot \mathbf D\mathbf w)\ge \theta >0$ for all $\mathbf w\in\mathbb S^{d-1}$. Therefore, for any $\mathbf x \ne \mathbf 0\in\mathbb R^d$, one has $\re (\mathbf x^T\mathbf D\mathbf x)=|\mathbf x|^2\re (\mathbf w\cdot \mathbf D\mathbf w)>0$, where $\mathbf x^T$ is the transpose of the vector $\mathbf x$.

Finally, since the condition $\mathsf S$ implies that for $\mathbf w\in \mathbb R^d$, there is an invertible matrix $S=S(\mathbf w)$ satisfying $S(\mathbf w)A(\mathbf w)=-A(\mathbf w)S(\mathbf w)$ and $S(\mathbf w)B=BS(\mathbf w)$, we obtain \eqref{eq:the 0-group 2} directly from Corollary \ref{prop:symmetry}. The proof is done.
\end{proof}

We study $E(i\mathbf k)=B+A(i\mathbf k)$ for large $\mathbf k\in\mathbb R^d$. Recall $E(\zeta,\mathbf w)=B+i\zeta A(\mathbf w)$ in \eqref{eq:radius-angular form of E(z)} for $(\zeta,\mathbf w)\in [0,+\infty)\times \mathbb S^{d-1}$. Note that under the assumption $\mathsf A$, there is an invertible matrix $R=R(\mathbf w)$ for $\mathbf w\in\mathbb S^{d-1}$ such that $R^{-1}AR$ is a diagonal matrix with nonzero entries are real linear eigenvalues of $A=A(\mathbf w)$ for $\mathbf w\in\mathbb S^{d-1}$. Hence, one can consider the $\ell$-th diagonal element of $R^{-1}AR$ as the linear function
\begin{equation}\label{eq:eigenvalues of A}
	\nu_\ell(\mathbf w):=\nu_\ell^{(0)}+\sum_{h=1}^d\nu_\ell^{(h)}w_h,\qquad \mathbf w=(w_1,\dots,w_d)\in\mathbb S^{d-1},
\end{equation}
where the coefficients $\nu_\ell^{(h)}\in\mathbb R$ for $h\in\{0,1,\dots,d\}$.
Let ${\boldsymbol \nu}_\ell:=(\nu_\ell^{(0)},\dots, \nu_\ell^{(d)})$ be the coefficient vector associated with $\nu_\ell$ for $\ell\in\{1,\dots,n\}$, one sets
\begin{equation*}%\label{eq:component 1 of the partition S}
	\mathcal S_1:=\{\ell\in\{1,\dots,n\}:{\boldsymbol \nu}_\ell={\boldsymbol \nu}_1\}.
\end{equation*}
For $i_j:=\min \{\{1,\dots,n\}\backslash \cup_{h=1}^{j-1} \mathcal S_h\}$, one defines
\begin{equation*}%\label{eq:component j of the partition S}
	\mathcal S_j:=\{\ell \in\{1,\dots,n\}:{\boldsymbol \nu}_\ell={\boldsymbol \nu}_{i_j}\}\qquad j=2,3,\dots
\end{equation*}
This procedure will stop at some finite integer $r\le n$ and $\mathcal S:=\{\mathcal S_1,\dots,\mathcal S_r\}$ is considered as a partition of $\{1,\dots,n\}$. One denotes by $[j]$ the representation of the elements of $\mathcal S_j$ and by $r_j$ the cardinality of $\mathcal S_j$ for $j\in\{1,\dots,r\}$.
\begin{lemma}[Measure-zero-set subtraction]\label{lem:the eigenvalues of A}
There is a Lebesgue measure zero set contained in $\mathbb S^{d-1}$ such that except for this set, the number of distinct eigenvalues of $A(\mathbf w)$ for $\mathbf w\in\mathbb S^{d-1}$ is $r$ and the algebraic multiplicities associated with them are $r_j$ for $j\in\{1,\dots,r\}$.
\end{lemma}
\begin{proof}
Recall the partition $\mathcal S=\{\mathcal S_1,\dots,\mathcal S_r\}$ with cardinality $r$. Assume that there are $i,j\in\{1,\dots,r\}$ such that $i\ne j$ and $\nu_{[i]}(\mathbf w_0)=\nu_{[j]} (\mathbf w_0)$ for some $\mathbf w_0\in\mathbb S^{d-1}$. We prove that $\mathbf w_0$ belongs to a Lebesgue measure zero set in $\mathbb R^{d-1}$. In fact, $\mathbf w_0$ belongs to the intersection of the affine hyperplane
\begin{equation*}
	(\nu_{[i]}^{(0)}-\nu_{[j]}^{(0)})+\sum_{h=1}^d(\nu_{[i]}^{(h)}-\nu_{[j]}^{(h)})x_h=0,\qquad (x_1,\dots,x_d)\in\mathbb R^d,
\end{equation*}
whose dimension is at most $d-1$ since the coefficient vectors ${\boldsymbol \nu}_{[i]}$ and ${\boldsymbol \nu}_{[j]}$ satisfy ${\boldsymbol \nu}_{[i]}\ne {\boldsymbol \nu}_{[j]}$ for any $i\ne j$ by definiton, and the unit sphere $\mathbb S^{d-1}$. Moreover, the dimension of the intersection is at most $d-2$ and it is therefore a Lebesgue measure zero set in $\mathbb R^{d-1}$. Thus, $\nu_{[i]}(\mathbf w)\ne \nu_{[j]}(\mathbf w)$ for any $i\ne j$ and for $\mathbf w\in\mathbb S^{d-1}$ subtracted a Lebesgue measure zero set. Finally, since the repeated eigenvalues of $A(\mathbf w)$ are $\nu_\ell (\mathbf w)$ determined by the coefficient vectors ${\boldsymbol \nu}_\ell$ for $\ell\in\{1,\dots,n\}$, it follows immediately that the number of distinct eigenvalues of $A(\mathbf w)$ for $\mathbf w\in\mathbb S^{d-1}$ is $r$ and the algebraic multiplicities associated with them are $r_j$, the cardinality of $\mathcal S_j$, for $j\in\{1,\dots,r\}$ excluding a Lebesgue measure zero set. We finish the proof.
\end{proof}

One sets, for $j\in\{1,\dots,r\}$, the projection
\begin{equation}\label{eq:the eigenprojection Pij0}
	(\Pi_j^{(0)})_{h\ell}:=
	\begin{cases}
		1& \textrm{if }h=\ell\in\mathcal S_j,\\
		0& \textrm{otherwise.}
	\end{cases}
\end{equation}
Let $R=R(\mathbf w)$ for $\mathbf w\in\mathbb S^{d-1}$ be the matrix satisfying the conditions $\mathsf A$ and $\mathsf R$. One has
\begin{proposition}[High-frequency approximation]
If the assumptions $\mathsf A$, $\mathsf R$ and $\mathsf D$ hold, then for large $\mathbf k\in\mathbb R^d$, $E(i\mathbf k)$ is almost everywhere approximated by
\begin{equation}\label{eq:E(z) for high frequency}
	E(i\mathbf k)=R\sum_{j=1}^r\sum_{m=1}^{s_j}\Upsilon_{jm}(i\mathbf k)\Pi_{jm}(i\mathbf k)R^{-1},
\end{equation}
where the constant $s_j\le r_j$ which is also constant as well as $r$, $\Upsilon_{jm}(i\mathbf k)$ commutes with $\Pi_{jm}(i\mathbf k)$ and one has
\begin{equation}\label{eq:Upsilonjm(z) for high frequency}
	\Upsilon_{jm}(i\mathbf k)=(\alpha_j(i\mathbf k)+\beta_{jm})I+\Theta_{jm}^{(0)}+\mathcal O(|\mathbf k|^{-1})
\end{equation}
and
\begin{equation}\label{eq:the total projection Pijm0}
	\Pi_{jm}(i\mathbf k)=\Pi_{jm}^{(0)}+\mathcal O(|\mathbf k|^{-1}),
\end{equation}
where $\alpha_j(i\mathbf k)=i|\mathbf k|\nu_{[j]}(\mathbf k/|\mathbf k|)$ for $\nu_{[j]}$ is in \eqref{eq:eigenvalues of A}, $\beta_{jm}$ with $\re \beta_{jm}>0$ is the $m$-th nonzero eigenvalue of $\Pi_j^{(0)}R^{-1}BR\Pi_j^{(0)}$ with the associated eigenprojection $\Pi_{jm}^{(0)}$ and eigennilpotent $\Theta_{jm}^{(0)}$.
\end{proposition}
\begin{proof}
Based on Lemma \ref{lem:the eigenvalues of A}, if the condition $\mathsf A$ holds, the spectrum of $R^{-1}AR(\mathbf w)$ for $\mathbf w\in\mathbb S^{d-1}$ is the set $\{\alpha_1(\mathbf w),\dots,\alpha_r(\mathbf w)\}$ where $\alpha_j(\mathbf w)=\nu_{[j]}(\mathbf w)$ given by \eqref{eq:eigenvalues of A} for $j\in\{1,\dots,r\}$ with finite constant $r$, the cardinality of $\mathcal S$, and $[j]$ is the representation of the elements of $\mathcal S_j$, for almost everywhere. Thus, from here in this proof, we consider always for almost everywhere and we drop $\mathbf w$ in the coefficients written in below if they are in fact constant for almost everywhere.

For $j\in\{1,\dots,r\}$, we study the $\alpha_j(\mathbf w)$-group of $\tilde E(\eta,\mathbf w)$ for $(\eta,\mathbf w)\in [0,+\infty)\times \mathbb S^{d-1}$, where
\begin{equation*}
	\tilde E(\eta,\mathbf w):=R^{-1}AR(\mathbf w)-i\eta R^{-1}BR,
\end{equation*}
for small $\eta>0$. One obtains from \eqref{eq:Total projection}, \eqref{eq:coefficients P1 and P2} and \eqref{eq:coefficients Xi} that the total projection $\Pi_j(\eta,\mathbf w)$ of the $\alpha_j(\mathbf w)$-group is approximated by
\begin{equation*}
	\Pi_j(\eta,\mathbf w)=\Pi_j^{(0)}(\mathbf w)+\mathcal O(\eta),
\end{equation*}
where $\Pi_j^{(0)}(\mathbf w)$ is the eigenprojection associated with $\alpha_j(\mathbf w)\in\sigma(R^{-1}AR(\mathbf w))$, the spectrum of $R^{-1}AR(\mathbf w)$. Moreover, by the definition of eigenprojection, if $\Gamma_j$ is an oriented closed curve in the resolvent set of $R^{-1}AR(\mathbf w)$ enclosing $\alpha_j(\mathbf w)$ except for the other eigenvalues of $R^{-1}AR(\mathbf w)$, then
\begin{align*}\label{eq:the computation of the eigenprojection Pij0}
	\Pi_j^{(0)}(\mathbf w)&=-\dfrac{1}{2\pi i}\int_{\Gamma_j}\diag \bigl(\nu_1(\mathbf w)-\mu )^{-1},\dots,(\nu_n(\mathbf w)-\mu)^{-1})\,\dif\mu \nonumber\\
	&=\diag \Bigl(-\dfrac{1}{2\pi i}\int_{\Gamma_j}(\nu_1(\mathbf w)-\mu )^{-1}\,d\mu,\dots,-\dfrac{1}{2\pi i}\int_{\Gamma_j}(\nu_n(\mathbf w)-\mu)^{-1}\,\dif\mu\Bigr)
\end{align*}
and it coincides \eqref{eq:the total projection Pijm0} since
\begin{equation*}
	-\dfrac{1}{2\pi i}\int_{\Gamma_j}(\nu_\ell (\mathbf w)-\mu )^{-1}\,\dif\mu=
	\begin{cases}
	1 &\textrm{if }  \ell\in\mathcal S_j,\\
	0 & \textrm{if } \ell\notin \mathcal S_j.
	\end{cases}
\end{equation*}
Thus, $\Pi_j^{(0)}(\mathbf w)$ is constant for almost everywhere and we can write $\Pi_j^{(0)}$ instead. On the other hand, since $\alpha_j(\mathbf w)$ is semi-simple, one has
\begin{equation}\label{eq:temp2}
	\ker(R^{-1}AR(\mathbf w)-\alpha_j(\mathbf w)I)=\ran \Pi_j^{(0)}.
\end{equation}

Therefore, by \eqref{eq:the division of T(z)P(z)} - \eqref{eq:subprojection Pj(z)} in Lemma \ref{prop:reduction process}, the formula of $\tilde E(\eta,\mathbf w)$ and \eqref{eq:temp2}, we have
\begin{equation*}
	\tilde E \Pi_j(\eta,\mathbf w)=\sum_{m=1}^{s_j}(\alpha_j(\mathbf w)I-i\eta\tilde E_{jm}(\eta,\mathbf w))\Pi_{jm}(\eta,\mathbf w),
\end{equation*}
where $\tilde E_{jm}(\eta,\mathbf w)$ commutes with $\Pi_{jm}(\eta,\mathbf w)$ and one has
\begin{equation*}
	\tilde E_{jm}(\eta,\mathbf w)=\beta_{jm}I+\Theta_{jm}^{(0)}+\mathcal O(\eta)
\end{equation*}
and
\begin{equation*}
	\Pi_{jm}(\eta,\mathbf w)=\Pi_{jm}^{(0)}+\mathcal O(\eta),
\end{equation*}
where $\beta_{jm}$ is the $m$-th eigenvalue of $\Pi_j^{(0)} R^{-1}BR\Pi_j^{(0)}$ considered in $\ran \Pi_j^{(0)}$ with the associated eigenprojection $\Pi_{jm}^{(0)}$ and eigennilpotent $\Theta_{jm}^{(0)}$ and $s_j$ is the number of such eigenvalues of $\Pi_j^{(0)} R^{-1}BR\Pi_j^{(0)}$. Note that $\beta_{jm}$, $\Pi_{jm}^{(0)}$, $\Theta_{jm}^{(0)}$ and $s_j$ are constant due to the fact that $R^{-1}BR$ is constant under the assumption $\mathsf R$ and $\Pi_j^{(0)}$ is constant for almost everywhere. Moreover, since one has
\begin{equation}\label{eq:temp3}
	\Pi_j^{(0)} R^{-1}BR\Pi_j^{(0)}(I-\Pi_j^{(0)})=O,
\end{equation}
where $O$ is the null matrix, {\it i.e.} $0$ is an eigenvalue of $\Pi_j^{(0)} R^{-1}BR\Pi_j^{(0)}$ considered in $\ran (I-\Pi_j^{(0)})$ with algebraic multiplicity $\dim (I-\Pi_j^{(0)})=n-\dim\ran \Pi_j^{(0)}$, it follows that $s_j\le r_j=\dim \ran \Pi_j^{(0)}$, the cardinality of $\mathcal S_j$, by definition.

Therefore, since $E(\zeta,\mathbf w)=i\zeta R\tilde E R^{-1}(\zeta^{-1},\mathbf w)$ where $\zeta^{-1}\to 0$ as $\zeta \to +\infty$, one obtains
\begin{equation}\label{eq:radius-angular E(z) for high frequency}
	E(\zeta,\mathbf w)=R\sum_{j=1}^r\sum_{m=1}^{s_j}\Upsilon_{jm}(\zeta,\mathbf w)\Pi_{jm}(\zeta,\mathbf w)R^{-1},
\end{equation}
where $\Upsilon_{jm}(\zeta,\mathbf w)$ commutes with $\Pi_{jm}(\zeta,\mathbf w)$ and
\begin{equation}\label{eq:radius-angular Upsilonjm(z) for high frequency}
	\Upsilon_{jm}(\zeta,\mathbf w)=(i\zeta \alpha_j(\mathbf w)+\beta_{jm})I+\Theta_{jm}^{(0)}+\mathcal O(\zeta^{-1})
\end{equation}
and
\begin{equation}\label{eq:radius-angular the total projection Pijm0}
	\Pi_{jm}(\zeta,\mathbf w)=\Pi_{jm}^{(0)}+\mathcal O(\zeta^{-1}).
\end{equation}

On the other hand, it then follows from \eqref{eq:radius-angular E(z) for high frequency} - \eqref{eq:radius-angular the total projection Pijm0} that for large $\zeta$, the eigenvalues of $E(\zeta,\mathbf w)$ are the eigenvalues of $\Upsilon_{jm}(\zeta,\mathbf w)$ and they are approximated by
\begin{equation*}
	\lambda_{jm}(\zeta,\mathbf w)=i\zeta\alpha_j(\mathbf w)+\beta_{jm}+{\scriptstyle \mathcal O}(1),
\end{equation*}
for $j\in\{1,\dots,r\}$ and $m\in\{1,\dots,s_j\}$. Thus, if the assumption $\mathsf D$ holds, then since $\alpha_j(\mathbf w)\in\mathbb R$, there is a constant $\theta >0$ such that for $0<\zeta^{-1}<\varepsilon$ small, one has
\begin{equation*}
	\theta /(1+\varepsilon^2 )\le \re\lambda_{jm}(\zeta,\mathbf w)\le \re \beta_{jm}+\varepsilon.
\end{equation*}
Let $\varepsilon\to 0$, one has $\re \beta_{jm}\ge \theta >0$. Moreover, one observes from \eqref{eq:temp3} that the nonzero eigenvalues of $\Pi_j^{(0)}R^{-1}BR\Pi_j^{(0)}$ always belong to the set of the eigenvalues of $\Pi_j^{(0)}R^{-1}BR\Pi_j^{(0)}$ considered in $\ran \Pi_j^{(0)}$. Thus, we can consider that $\beta_{jm}$ are the nonzero eigenvalues of $\Pi_j^{(0)}R^{-1}BR\Pi_j^{(0)}$ without specifying that they are considered in $\ran \Pi_j^{(0)}$ or not. We finish the proof by writing \eqref{eq:radius-angular E(z) for high frequency} - \eqref{eq:radius-angular the total projection Pijm0} in the coordinates $\mathbf k\in\mathbb R^d$.
\end{proof}
\begin{remark}[Intermediate-frequency approximation]
In this paper, we will not use any expansions of $E(i\mathbf k)=E(\zeta,\mathbf w)$ in the intermediate frequency but note that there is a finite number of exceptional curves of $E(\zeta,\mathbf w)$ for $0<\zeta<+\infty$ in general. In the domain excluding these curves, the number of distinct eigenvalues of $E(\zeta,\mathbf w)$ and their algebraic multiplicities are constant (see \cite{kato,bianchini07}).
\end{remark}

\section{Decay estimates (Core of the paper)}\label{sec:Decay estimates}
In this section, we prove the estimates used in the proofs of Theorem \ref{theo:LpLq estimates 1} and Theorem \ref{theo:LpLq estimates 2}.  We primarily give a priori estimates for the principal parabolic part of the fundamental solution $\Gamma_t$ to the system \eqref{eq:the dissipative hyperbolic equation}. Then, we estimate $\Gamma_t$ by diving the frequency space into: the low frequency, the intermediate frequency and the high frequency. The main proofs are related to the interpolation between the $L^\infty$-$L^1$ estimate and the $L^p$-$L^p$ estimate for $1\le p\le \infty$.
Moreover, the $L^\infty$-$L^1$ estimate is obtained directly while the $L^p$-$L^p$ estimate is obtained based on the Carlson--Beurling inequality \eqref{eq:Carlson-Beurling Hs} in Lemma \ref{lem:Carlson-Beurling}. Moreover, since the Carlson--Beurling inequality \eqref{eq:Carlson-Beurling Hs} depends on the analysis of partial derivatives, one considers the followings.

\vskip.15cm
Let $\mathcal I$ be an index-set given by $\mathcal I:=\{i_1,\dots,i_s\}$ with possible repeated indices $i_\ell\in\{1,\dots,d\}$ {\it i.e.} we allows $i_h=i_\ell$ for some $h\ne \ell$. For any partition $\{\mathcal I_j:j=1,\dots,r\}$ of $\mathcal I$ where $\mathcal I_j:=\{i_1^j,\dots,i_{s_j}^j\}$ for some $r\in\{1,\dots,s\}$, one defines the partial derivative $\partial_{\mathcal I_j}$ in $\mathbf x\in\mathbb R^d$ of scalar smooth functions $q(\mathbf x,t)$ on $\mathbb R^d\times \mathbb R_+$ by
\begin{equation*}
	\partial_{\mathcal I_j}q(\mathbf x,t):=\partial^{s_j}_{x_{i_1^j}\dots x_{i_{s_j}^j}}q(\mathbf x,t),
\end{equation*}
which is the usual partial derivative. Note that, once considering a partition, we do not consider any $\mathcal I_j=\emptyset$, and thus, $s_j\ge 1$ for all $j$. On the other hand, for any fixed $\alpha\in\mathbb N^d$, if $|\alpha|=0$ {\it i.e.} $\alpha=0$, we set $\mathcal I_\alpha:=\emptyset$ and $|\mathcal I_\alpha|:=0$. If $|\alpha|=s\in \mathbb Z_+$, $\alpha$ determines an index-set $\mathcal I_\alpha=\{i_1,\dots,i_s\}\ne\emptyset$ with possible repeated indices. In fact, if $\alpha=(\alpha_1,\dots,\alpha_d)$, we can define the index-set $\mathcal I_\alpha$ having $\alpha_\ell$ indices $\ell\in\{1,\dots,d\}$. We also set $|\mathcal I_\alpha|:=s\ge 1$ and $|\mathcal I_j|:=s_j\ge 1$ for $j\in\{1,\dots,r\}$ if $\mathcal I_\alpha\ne \emptyset$.
\begin{lemma}[Partial derivative]\label{lem:derivatives}
Let $\alpha \in \mathbb N^d$ with $|\alpha|\ge 0$, for any scalar smooth functions $q=q(\mathbf x,t)$ on $\mathbb R^d\times \mathbb R_+$, we have
\begin{equation}\label{eq:derivatives}
	\partial^\alpha e^{q(\mathbf x,t)}=\sum_{\{\mathcal I_j:j=1,\dots,r\},r\le |\alpha|}\partial_{\mathcal I_1}q(\mathbf x,t)\dots \partial_{\mathcal I_r} q(\mathbf x,t)e^{q(\mathbf x,t)},
\end{equation}
where $\{\mathcal I_j:j=1,\dots,r\}$ is any possible partition of the index-set $\mathcal I_\alpha$ determined by $\alpha$.
\end{lemma}
\begin{proof}
We prove by induction. Let $\alpha\in\mathbb N^d$, if $|\alpha|=0$, then since $\mathcal I_\alpha=\emptyset$, there is no partition of $\mathcal I_\alpha$ to be considered, and thus, $\partial^0e^{q(\mathbf x,t)}=e^{q(\mathbf x,t)}$. If $|\alpha|=1$, by the definition of $\partial^\alpha$, we have
\begin{equation}\label{eq:derivatives s=1}
	\partial^\alpha e^{q(\mathbf x,t)}=\partial_{x_i}^1e^{q(\mathbf x,t)}=\partial_{x_i}^1q(\mathbf x,t)e^{q(\mathbf x,t)}
\end{equation}
if $\alpha_i=1$ and $\alpha_\ell=0$ for all $\ell\ne i$. On the other hand, the index-set determined by $\alpha$ in this case is $\mathcal I_\alpha=\{i\}$ since $\alpha_i=1$. Thus $\mathcal I_\alpha$ has only one possible partition which is itself and \eqref{eq:derivatives s=1} coincides \eqref{eq:derivatives}.

Given an integer $s\ge 1$, assume that \eqref{eq:derivatives} holds for any $\alpha\in\mathbb N^d$ satisfying $|\alpha|=s$. For any $\beta\in\mathbb N^d$ with $|\beta|=s+1$, $\beta=(\alpha_1,\dots,\alpha_i+1,\dots,\alpha_d)$ for some $\alpha=(\alpha_1,\dots,\alpha_d)\in\mathbb N^d$ and $i=1,\dots,d$. Hence, we have
\begin{align}\label{eq:derivatives s+1}
	\partial^\beta e^{q(\mathbf x,t)}=\partial_{x_i}^1\partial^\alpha e^{q(\mathbf x,t)}&=\sum_{\{\mathcal I_j:j=1,\dots,r\}}\sum_{\ell=1}^r\partial_{\mathcal I_1}q(\mathbf x,t)\dots\partial_{x_i}^1\partial_{\mathcal I_\ell}q(\mathbf x,t)\dots \partial_{\mathcal I_r} q(\mathbf x,t)e^{q(\mathbf x,t)} \nonumber \\
	&\hskip2cm+\sum_{\{\mathcal I_j:j=1,\dots,r\}}\partial_{\mathcal I_1}q(\mathbf x,t)\dots \partial_{\mathcal I_r} q(\mathbf x,t)\partial_{x_i}^1q(\mathbf x,t)e^{q(\mathbf x,t)},
\end{align}
where $\{\mathcal I_j:j=1,\dots,r\}$ is any possible partition of the index-set $\mathcal I_\alpha$ determined by $\alpha$.
We then consider all of possible partitions of $\mathcal I_\beta$. The first possibilities are the partitions $\{\{\mathcal I_j:j=1,\dots,r\}, \{i\}\}$ since $\mathcal I_\beta$ has $\alpha_i+1$ indices $i$. The last choices are that for each partition $\{\mathcal I_j:j=1,\dots,r\}$ of $\mathcal I_\alpha$, we generate the partition $\{\mathcal I_j':j=1,\dots,r\}$ of $\mathcal I_\beta$ by putting $i$ into $\mathcal I_\ell$ and let $\mathcal I_j'=\mathcal I_j$ for all $j\ne \ell$ for $\ell\in\{1,\dots,r\}$. Thus, since $r$ varies, there is no other possible partition of $\mathcal I_\beta$ to take part in. Therefore, we obtain from \eqref{eq:derivatives s+1} that
\begin{equation*}
	\partial^\beta e^{q(\mathbf x,t)}=\sum_{\{\mathcal I_j':j=1,\dots,r'\}}\partial_{\mathcal I_1'}q(\mathbf x,t)\dots \partial_{\mathcal I_r'} q(\mathbf x,t)e^{q(\mathbf x,t)},
\end{equation*}
where the sum is made on all possible partitions $\{\mathcal I_j':j=1,\dots,r'\}$ of $\mathcal I_\beta$ determined by $\beta$. We thus proved \eqref{eq:derivatives}.
\end{proof}
\begin{remark}
Lemma \ref{lem:derivatives} is applied only to the case where $q=q(\mathbf x,t)$ is scalar for $(\mathbf x,t)\in\mathbb R^d\times \mathbb R_+$, the matrix case is a challenge as the loss of commutativity of $q$ and its partial derivatives.  
\end{remark}
\begin{proposition}[Parabolic estimate]\label{prop:decay estimate parabolic}
If $\mathbf D\in\mathbb R^{d\times d}$ is positive definite, for $1\le q\le p\le \infty $, there is a constant $C>0$ such that for any $U_0\in L^q(\mathbb R^d)$, one has
\begin{equation}\label{eq:decay estimate parabolic}
	\|\mathcal F^{-1}(e^{-\mathbf k\cdot \mathbf D\mathbf k t})*U_0\|_{L^p}\le C(1+t)^{-\frac d2(\frac 1q-\frac1p)}\|U_0\|_{L^q},\qquad\forall t>0.
\end{equation}
\end{proposition}
\begin{proof}
We primarily study the $L^\infty$-$L^1$ estimate. By the Young inequality and since $\mathbf D$ is positive definite, there are constants $c>0$ and $C>0$ such that for $t>0$, we have
\begin{align}\label{eq:Linfty-L1 bound for parabolic}
	\|\mathcal F^{-1}(e^{-\mathbf k\cdot \mathbf D\mathbf k t})*U_0\|_{L^\infty}&\le C\|\mathcal F^{-1}(e^{-\mathbf k\cdot \mathbf D\mathbf k t})\|_{L^\infty}\|U_0\|_{L^1}\nonumber\\
	&\le C\|e^{-c|\cdot|^2t}\|_{L^1}\|U_0\|_{L^1}\le C(1+t)^{-\frac d2}\|U_0\|_{L^1}.
\end{align}

We study the $L^p$-$L^p$ estimate for $1\le p\le \infty$. Let $\alpha\in \mathbb N^d$ with $|\alpha|\ge 0$, by the formula \eqref{eq:derivatives} in Lemma \ref{lem:derivatives}, we have
\begin{equation*}
	\partial^\alpha(e^{-\mathbf k\cdot \mathbf D\mathbf kt})=\sum_{\{\mathcal I_j:j=1,\dots,r\}}\partial_{\mathcal I_1}(-\mathbf k\cdot \mathbf D\mathbf k t)\dots \partial_{\mathcal I_r}(-\mathbf k\cdot \mathbf D\mathbf k t)e^{-\mathbf k\cdot \mathbf D\mathbf kt},
\end{equation*}
where $\{\mathcal I_j:j=1,\dots,r\}$ is any possible partition of the index-set $\mathcal I_\alpha$ determined by $\alpha$.

On the other hand, by the definition of $\partial_{\mathcal I_j}$, there is a constant $C>0$ such that
\begin{equation*}
	|\partial_{\mathcal I_j}(-\mathbf k\cdot \mathbf D\mathbf kt)|\le C\cdot \begin{cases}
	0&\textrm{if }|\mathcal I_j|>2,\\
	t &\textrm{if }|\mathcal I_j|=2,\\
	|\mathbf k|t & \textrm{if }|\mathcal I_j|=1,
	\end{cases}
\end{equation*}
where $|\mathcal I_j|$ is the number of elements of $\mathcal I_j$ with possible repeated indices for $j\in\{1,\dots,r\}$. We are then not interested in the cases where $|\mathcal I_j|>2$ for some $j\in\{1,\dots,r\}$. Thus, we can consider only the partitions $\{\mathcal I_j:j=1,\dots,r\}$ of $\mathcal I_\alpha$ where $1\le |\mathcal I_j|\le 2$. Hence, we have
\begin{equation*}
	|\partial_{\mathcal I_1}(-\mathbf k\cdot \mathbf D\mathbf k t)|\dots |\partial_{\mathcal I_r}(-\mathbf k\cdot \mathbf D\mathbf k t)|\le C|\mathbf k|^{m}t^{m+\ell},
\end{equation*}
where $m\ge 0$ is the cardinality of the set $\{j\in\{1,\dots,r\}:|\mathcal I_j|=1\}$ and $\ell\ge 0$ is the cardinality of the set $\{j\in\{1,\dots,r\}:|\mathcal I_j|=2\}$. Moreover, by definition, one has $m+2\ell=|\mathcal I_\alpha|=|\alpha|$, where $|\mathcal I_\alpha|=\sum_{j=1}^r|\mathcal I_j|$, the number of elements of the index-set $\mathcal I_\alpha$ determined by $\alpha$ with possible repeated indices.

Thus, since $\mathbf D$ is positive definite, there are constants $c>0$ and $C>0$ such that
\begin{equation}\label{eq:derivatives estimate parabolic previous}
	|\partial^\alpha e^{-\mathbf k\cdot \mathbf D\mathbf kt}|\le C\sum_{\substack{\{\mathcal I_j:j=1,\dots,r\},r\le |\alpha|\\
	1\le |\mathcal I_j|\le 2}}|\mathbf k|^mt^{m+\ell}e^{-c|\mathbf k|^2t}.
\end{equation}
Hence, since $m+2\ell=|\alpha|$, we have
\begin{align}\label{eq:derivatives estimate parabolic}
	\|\partial^\alpha e^{-\mathbf k\cdot \mathbf D\mathbf kt}\|_{L^2}^2&\le C\sum_{\substack{\{\mathcal I_j:j=1,\dots,r\},r\le |\alpha|\\
	1\le |\mathcal I_j|\le 2}}\int_{\mathbb R^d} |\mathbf k|^{2m}t^{2m+2\ell}e^{-2c|\mathbf k|^2t}\,\dif\mathbf k\nonumber\\
	&\le C(1+t)^{m+2\ell -\frac d2}=  C(1+t)^{|\alpha|-\frac d2}.
\end{align}

By the Carlson--Beurling inequality \eqref{eq:Carlson-Beurling Hs} in Lemma \ref{lem:Carlson-Beurling}, one has
\begin{align*}
	\|e^{-\mathbf k\cdot \mathbf D\mathbf kt}\|_{M_p}&\le C\|e^{-\mathbf k\cdot \mathbf D\mathbf kt}\|_{L^2}^{1-\frac{d}{2s}}\Bigl(\sum_{|\alpha|=s}\|\partial^\alpha e^{-\mathbf k\cdot \mathbf D\mathbf kt}\|_{L^2}\Bigr)^{\frac{d}{2s}}\nonumber\\
	&\le C(1+t)^{-\frac{d}{4}(1-\frac{d}{2s})+(\frac s2-\frac d4)\frac{d}{2s}}\le C,
\end{align*}
for any integer $s>d/2$, $1\le p\le \infty$ and $t>0$. Therefore, by the definition of the $M_p$-norm, we have the $L^p$-$L^p$ estimate
\begin{equation}\label{eq:Lp-Lp bound for parabolic}
	\|\mathcal F^{-1}(e^{-\mathbf k\cdot \mathbf D\mathbf kt})*U_0\|_{L^p}\le C\|U_0\|_{L^p},\qquad \forall 1\le p\le \infty.
\end{equation}

Finally, by applying the interpolation inequality and the estimates \eqref{eq:Linfty-L1 bound for parabolic} and \eqref{eq:Lp-Lp bound for parabolic}, we obtain \eqref{eq:decay estimate parabolic}. The proof is done.
\end{proof}
\begin{remark}\label{rem:derivatives estimate parabolic previous}
Note that the derivative estimate \eqref{eq:derivatives estimate parabolic previous} is true for all $\mathbf k\in\mathbb R^d$.
\end{remark}

Let $\chi_j$ for $j=1,2,3$ be cut-off functions on $\mathbb R^d$, valued in $[0,1]$, such that $\textrm{supp }\chi_1\subset\{\mathbf k\in\mathbb R^d:|\mathbf k|\le\varepsilon\}$ and $\textrm{supp }\chi_3\subset\{\mathbf k\in\mathbb R^d: |\mathbf k|\ge \rho\}$ for small $\varepsilon>0$ and large $\rho>0$, and $\chi_2(\mathbf k):=1-\chi_1(\mathbf k)-\chi_3(\mathbf k)$ for $\mathbf k\in\mathbb R^d$. We are now going to study the large-time behavior of the fundamental solution $\Gamma_t$ to the system \eqref{eq:the dissipative hyperbolic equation} in each partition of the frequency space.
\vskip.15cm
 For $\mathbf k\in\mathbb R^d$, we recall the Fourier transform of the fundamental solution $\Gamma_t$ to the system \eqref{eq:the dissipative hyperbolic equation}, namely
\begin{equation}\label{eq:Fourier transform of Gamma kernel}
	\hat\Gamma_t(\mathbf k)=e^{-E(i\mathbf k)t},
\end{equation}
where $E$ is given in \eqref{eq:operators E and A}. We also recall
\begin{equation}\label{eq:Fourier transform of Phi and Psi kernel}
	\hat \Phi_t(\mathbf k)=e^{-\mathbf c\cdot i\mathbf kt-\mathbf k\cdot \mathbf D\mathbf kt}P_0^{(0)},\qquad \hat \Psi_t(\mathbf k)=e^{-\mathbf k\cdot \mathbf D\mathbf kt}(P_0^{(0)}+\mathbf P_0^{(1)}\cdot i\mathbf k),
\end{equation}
where $\mathbf c,\,\mathbf D$ are given by \eqref{eq:coefficients c and D 1}, $P_0^{(0)}$ is given by \eqref{eq:coefficients P0 and S0 1} and $\mathbf P_0^{(1)}$ is given by \eqref{eq:coefficient P01 of the total projection P0 1}.

\subsection{Low-frequency analysis}
The aim of this subsection is to study the $L^p$-$L^q$ estimate for the low-frequency part of $\Gamma_t$ for any $1\le q\le p\le \infty$. One thus considers $\hat \Gamma_t\chi_1$.

\begin{lemma}[Derivative estimate]\label{lem:derivatives estimate}
Let $p(\mathbf x)$ be a scalar polynomial on $\mathbb R^d$ such that the lowest order of $p(\mathbf x)$ is $h\ge 1$ and let $\alpha \in \mathbb N^d$ with $|\alpha|\ge 0$. There is a constant $C>0$ such that for small $\mathbf x\in\mathbb R^d$ and $t>0$, we have
\begin{equation}\label{eq:derivatives estimate}
	|\partial^\alpha e^{p(\mathbf x)t}|\le C\sum_{\{\mathcal I_j:j=1,\dots,r\},\,r\le |\alpha|} |\mathbf x |^{\sum_{k=1}^{h-1}km_k } t^{\ell+\sum_{k=0}^{h-1}m_k} |e^{p(\mathbf x)t}|,
\end{equation}
where the integer $m_k\ge 0$ is the cardinality of $\{j\in\{1,\dots,r\}:|I_j|=h-k\}$ for each $k\in\{0,\dots,h-1\}$ and the integer $\ell\ge 0$ satisfies
\begin{equation}\label{eq:restriction of n}
	h\ell <  |\alpha|-\sum_{k=0}^{h-1}(h-k)m_k
\end{equation}
and $\{\mathcal I_j:j=1,\dots,r\}$ is any possible partition of the index-set $\mathcal I_\alpha$ determined by $\alpha$.
\end{lemma}
\begin{proof}
Let $\alpha\in \mathbb N^d$ with $|\alpha|\ge 0$ and $p(\mathbf x)$ be a polynomial on $\mathbb R^d$ such that the lowest order of $p(\mathbf x)$ is $h\ge 1$. For any partition $\{\mathcal I_j:j=1,\dots,r\}$ of $\mathcal I_\alpha$ determined by $\alpha$, by the definition of $\partial_{\mathcal I_j}$, there is a constant $C(j)>0$ such that
\begin{equation*}
	|\partial_{\mathcal I_j}p(\mathbf x)|\le C(j)\cdot\begin{cases} 1& \textrm{if }|\mathcal I_j|\ge h,\\
	|\mathbf x|^{k}&\textrm{if }|\mathcal I_j|=h-k,\end{cases}
\end{equation*}
for any $k\in\{0,\dots,h-1\}$ and small $\mathbf x\in\mathbb R^d$, where $|\mathcal I_j|$ is the number of elements of the index-set $\mathcal I_j$ with possible repeated indices. Note that $\sum_{j=1}^r|\mathcal I_j|=|\mathcal I_\alpha|=|\alpha|$ by definition. It implies that there is a constant $C(r)=\max_{j}C(j)>0$ such that for small $\mathbf x\in \mathbb R^d$ and $t>0$, we have
\begin{equation}\label{eq:partial derivatives estimate}
	|\partial_{\mathcal I_1}(p(\mathbf x)t)|\dots|\partial_{\mathcal I_r}(p(\mathbf x)t)|\le C(r)|\mathbf x|^{\sum_{k=1}^{h-1}km_k}t^{\ell+\sum_{k=0}^{h-1}m_k},
\end{equation}
where $m_k\ge 0$ is the cardinality of $\{j\in\{1,\dots,r\}:|\mathcal I_j|=h-k\}$ for $k\in\{0,\dots,h-1\}$ and $\ell\ge 0$ is the cardinality of $\mathcal J:=\{j\in\{1,\dots,r\}:|\mathcal I_j|>h\}$. Moreover, we have
\begin{equation}\label{eq:restriction of n 2}
	h\ell < \sum_{j\in \mathcal J}|\mathcal I_j|=|\alpha|-\sum_{k=0}^{h-1}(h-k)m_k.
\end{equation}
We thus obtain \eqref{eq:derivatives estimate} and \eqref{eq:restriction of n} with $C=\max_rC(r)>0$ from \eqref{eq:derivatives}, \eqref{eq:partial derivatives estimate} and \eqref{eq:restriction of n 2}. The proof is done.
\end{proof}
Let $P_0$ be given by \eqref{eq:the total projection P0 of the 0-group}, we have the following.
\begin{proposition}[Low-frequency estimate]\label{prop:low-frequency estimates}
If the assumptions $\mathsf B$ and $\mathsf D$ hold, then for $1\le q\le p\le \infty $, there is a constant $C>0$ such that for $t>0$, we have
\begin{equation}\label{eq:low frequency estimate 1}
	\|\mathcal F^{-1}((\hat \Gamma_t(\mathbf k)P_0(i\mathbf k)-\hat \Phi_t(\mathbf k))\chi_1(\mathbf k))*u_0\|_{L^p}\le C(1+t)^{-\frac d2(\frac 1q-\frac1p)-\frac 12}\|u_0\|_{L^q}.
\end{equation}
If the condition $\mathsf S$ holds in addition, then we have
\begin{equation}\label{eq:low frequency estimate 2}
	\|\mathcal F^{-1}((\hat \Gamma_t(\mathbf k)P_0(\mathbf k)-\hat \Psi_t(\mathbf k))\chi_1(\mathbf k))*u_0\|_{L^p}\le C(1+t)^{-\frac d2(\frac 1q-\frac1p)-1}\|u_0\|_{L^q}.
\end{equation}
On the other hand, for $1\le q\le 2\le  p\le \infty $, there are constants $c>0$ and $C>0$ such that for $t>0$, we have
\begin{equation}\label{eq:low frequency estimate 3}
	\|\mathcal F^{-1}(\hat \Gamma_t(\mathbf k)(I-P_0(i\mathbf k))\chi_1(\mathbf k))*u_0\|_{L^p}\le Ce^{-ct}\|u_0\|_{L^q}.
\end{equation}
Moreover, \eqref{eq:low frequency estimate 3} holds for $1\le q\le p\le \infty$ and $t\ge 1$ if $\Gamma_t$ has compact support contained in $\{(\mathbf x,t)\in \mathbb R^d\times \mathbb R:|\mathbf x/t|\le C\}$ for some constant $C>0$.
\end{proposition}
\begin{proof}
Under assumptions $\mathsf B$ and $\mathsf D$, from \eqref{eq:E(z) for low frequency} - \eqref{eq:the total projection Pj of the other groups} in Proposition \ref{prop:low-frequency}, for small $\mathbf k\in\mathbb R^d$, one has
\begin{equation}\label{eq:decomposition of solution for low frequency 1new}
	\hat \Gamma_t(\mathbf k)\chi_1(\mathbf k)=\hat \Gamma_t^{(1)}(\mathbf k)\chi_1(\mathbf k)+\hat \Gamma_t^{(2)}(\mathbf k)\chi_1(\mathbf k),
\end{equation}
where
\begin{equation}\label{eq:decomposition of solution for low frequency 1}
	\hat \Gamma_t^{(1)}(\mathbf k)=e^{-\lambda_0(i\mathbf k)t}P_0(i\mathbf k)=e^{-\mathbf c\cdot i\mathbf k t-\mathbf k\cdot \mathbf D\mathbf k t+\mathcal O(|\mathbf k|^3)t}(P_0^{(0)}+\mathcal O(|\mathbf k|))
\end{equation}
and
\begin{equation}\label{eq:decomposition of solution for low frequency 2}
	\hat \Gamma_t^{(2)}(\mathbf k)=\sum_{j=1}^se^{-E_j(i\mathbf k)t}P_j(i\mathbf k)=\sum_{j=1}^se^{-\lambda_j^{(0)}t}e^{-N_j^{(0)}t+\mathcal O(|\mathbf k|)t}(P_j^{(0)}+\mathcal O(|\mathbf k|)),
\end{equation}
where $\mathbf c\in\mathbb R^d$ and $\mathbf D\in\mathbb R^{d\times d}$ is positive definite given by \eqref{eq:coefficients c and D}, $P_0^{(0)}$ is the eigenprojection associated with $0\in\sigma(B)$, and $\lambda_j^{(0)}\in\sigma(B)\backslash \{0\}$, $\re \lambda_j^{(0)}>0$, with the associated eigenprojection $P_j^{(0)}$ and eigennilpotent $N_j^{(0)}$ for $j\in\{1,\dots,s\}$ and $s$ is the cardinality of $\sigma(B)\backslash \{0\}$.

\vskip.25cm
We now prove Proposition \ref{prop:low-frequency estimates} by primarily establishing the $L^\infty$-$L^1$ estimate. Then, by constructing the $L^p$-$L^p$ estimate for $1\le p\le \infty$, we apply the interpolation inequality. 
\vskip.25cm
{\it Step 1. $L^\infty$-$L^1$ estimate.}
\vskip.25cm
By changing the coordinates $(\mathbf x,t)\mapsto (\mathbf x-\mathbf c t,t)$, one can always assume that $\mathbf c=\mathbf 0$ without loss of generality. We study the $L^\infty$-$L^1$ estimate. Consider
\begin{equation*}
	\hat \Phi_t(\mathbf k)=e^{-\mathbf k\cdot \mathbf D\mathbf k t}P_0^{(0)},
\end{equation*}
one has
\begin{equation}\label{eq:decomposition of solution for low frequency 1 into I and J}
	(\hat \Gamma_t^{(1)}(\mathbf k)-\hat \Phi_t(\mathbf k))\chi_1(\mathbf k)=I+J,
\end{equation}
where
\begin{equation}\label{eq:components I and J for low frequency}
	I:=e^{-\mathbf k\cdot \mathbf D\mathbf k t}(e^{\mathcal O(|\mathbf k|^3)t}-1)P_0^{(0)}\chi_1(\mathbf k),\qquad J:=e^{-\mathbf k\cdot \mathbf D\mathbf k t+\mathcal O(|\mathbf k|^3)t}\mathcal O(|\mathbf k|)\chi_1(\mathbf k).
\end{equation}
Then, there are constants $c>0$ and $C>0$ such that
\begin{equation*}
	|(\hat \Gamma_t^{(1)}(\mathbf k)-\hat \Phi_t(\mathbf k))\chi_1(\mathbf k)|\le |I|+|J|\le Ce^{-c|\mathbf k|^2t}(|\mathbf k|^3t+|\mathbf k|).
\end{equation*}
Thus, we have
\begin{equation*}
	\|(\hat \Gamma_t^{(1)}-\hat \Phi_t)\chi_1\|_{L^1}\le C\int_{\mathbb R^d}e^{-c|\mathbf k|^2t}(|\mathbf k|^3t+|\mathbf k|)\,\dif \mathbf k\le C(1+t)^{-\frac d2-\frac 12}.
\end{equation*}
By the Young inequality, we have
\begin{align*}
	\|\mathcal F^{-1}((\hat \Gamma_t^{(1)}(\mathbf k)-\hat \Phi_t(\mathbf k))\chi_1(\mathbf k))*u_0\|_{L^\infty}&\le \|(\hat \Gamma_t^{(1)}-\hat \Phi_t)\chi_1\|_{L^1}\|u_0\|_{L^1}\nonumber\\
	&\le C(1+t)^{-\frac d2-\frac 12}\|u_0\|_{L^1}.
\end{align*}

Recall $\hat \Gamma_t^{(2)}(\mathbf k)$ in \eqref{eq:decomposition of solution for low frequency 2}, one has
\begin{equation*}
	\hat \Gamma_t^{(2)}(\mathbf k)=\sum_{j=1}^se^{-E_j(i\mathbf k)t}P_j(i\mathbf k)=\sum_{j=1}^se^{-\lambda_j^{(0)}t}e^{-N_j^{(0)}t+\mathcal O(|\mathbf k|)t}(P_j^{(0)}+\mathcal O(|\mathbf k|)),
\end{equation*}
where $\lambda_j^{(0)}\in\sigma(B)\backslash \{0\}$, $\re \lambda_j^{(0)}>0$, with the associated eigenprojection $P_j^{(0)}$ and eigennilpotent $N_j^{(0)}$ for $j\in\{1,\dots,s\}$ and $s$ is the cardinality of $\sigma(B)\backslash \{0\}$. Thus, by the Householder theorem which is Theorem 7.1 p. 133 in \cite{serre10}, for any $\varepsilon>0$, there is an induced norm such that $|N_j^{(0)}|\le \varepsilon$ and due to the fact that every norms in finite-dimensional space are equivalent, one deduces that since $|\mathbf k|$ small and $\re \lambda_j^{(0)}>0$ for all $j\in\{1,\dots,s\}$, there are constants $c,c'>0$ and $C>0$ such that
\begin{equation}\label{eq:point-wise estimate for Gamma_t(2)}
	|\hat \Gamma_t^{(2)}(\mathbf k)\chi_1(\mathbf k)|\le C\sum_{j=1}^se^{-\re \lambda_j^{(0)}t}e^{\varepsilon t+c'|\mathbf k|t}|\chi_1(\mathbf k)|\le Ce^{-ct}|\chi_1(\mathbf k)|.
\end{equation}
Hence, we obtain
\begin{equation*}
	\|\hat \Gamma_t^{(2)}\chi_1\|_{L^1}\le Ce^{-ct}\|\chi_1(\mathbf k)\|_{L^1}\le Ce^{-ct}.
\end{equation*}
It implies that
\begin{equation*}
	\|\mathcal F^{-1}(\hat \Gamma_t^{(2)}(\mathbf k)\chi_1(\mathbf k))*u_0\|_{L^\infty}\le \|\hat \Gamma_t^{(2)}\chi_1\|_{L^1}\|u_0\|_{L^1}\le Ce^{-ct}\|u_0\|_{L^1}.
\end{equation*}

Therefore, the $L^\infty$-$L^1$ estimate holds, namely
\begin{align}\label{eq:Linfty-L1 bound}
	\|\mathcal F^{-1}((\hat \Gamma_t(\mathbf k)P_0(\mathbf k)-\hat \Phi_t(\mathbf k))\chi_1(\mathbf k))*u_0\|_{L^\infty}&= \|\mathcal F^{-1}((\hat \Gamma_t^{(1)}(\mathbf k)-\hat \Phi_t(\mathbf k))\chi_1(\mathbf k))*u_0\|_{L^\infty}\nonumber\\
	&\le C(1+t)^{-\frac d2-\frac 12}\|u_0\|_{L^1}
\end{align}
and
\begin{equation}\label{eq:Linfty-L1 bound Gamma2}
	\|\mathcal F^{-1}(\hat \Gamma_t(\mathbf k)(I-P_0(\mathbf k))\chi_1(\mathbf k))*u_0\|_{L^\infty}=  \|\mathcal F^{-1}(\hat \Gamma_t^{(2)}(\mathbf k)\chi_1(\mathbf k))*u_0\|_{L^\infty}\le Ce^{-ct}\|u_0\|_{L^1}.
\end{equation}
\vskip.25cm
{\it Step 2. $L^p$-$L^p$ estimates.}
\vskip.25cm
We study the $L^p$-$L^p$ estimate for $1\le p\le \infty$ by Lemma \ref{lem:Carlson-Beurling}. In the spirit of Lemma \ref{lem:Carlson-Beurling}, we need to estimate the $L^2$-norm of $\partial^\alpha((\hat \Gamma_t-\hat \Phi_t)\chi_1)$ for $\alpha\in\mathbb N^d$.

Recall the decomposition $\hat \Gamma_t\chi_1=\hat \Gamma_t^{(1)}\chi_1+\hat \Gamma_t^{(2)}\chi_1$ in \eqref{eq:decomposition of solution for low frequency 1new}, where $\hat \Gamma_t^{(1)}$ is given by \eqref{eq:decomposition of solution for low frequency 1} and $\hat \Gamma_t^{(2)}$ is given by \eqref{eq:decomposition of solution for low frequency 2}. We primarily estimate the $L^2$-norm of $\partial^\alpha((\hat \Gamma_t^{(1)}-\hat \Phi_t)\chi_1)$ for any $\alpha\in \mathbb N^d$ by considering the decomposition $(\hat \Gamma_t^{(1)}-\hat \Phi_t)\chi_1=I+J$ in \eqref{eq:decomposition of solution for low frequency 1 into I and J}, where $I$ and $J$ are given by \eqref{eq:components I and J for low frequency}. By the Leibniz rule, one has
\begin{equation}\label{eq:derivatives of I}
	\partial^\alpha I=\sum_{\nu\le \alpha}\sum_{\tau\le \nu}\begin{pmatrix} \alpha \\ \nu \end{pmatrix}\begin{pmatrix} \nu \\ \tau  \end{pmatrix}\partial^\tau e^{-\mathbf k\cdot \mathbf D\mathbf k t}\partial^{\nu-\tau}(e^{\mathcal O(|\mathbf k|^3)t}-1)\partial^{\alpha-\nu}\chi_1(\mathbf k)=I^{(1)}+I^{(2)}	,
\end{equation}
where
\begin{equation}\label{eq:I(1)}
	I^{(1)}:=\sum_{\nu\le \alpha}\begin{pmatrix} \alpha \\ \nu \end{pmatrix}\partial^\nu e^{-\mathbf k\cdot \mathbf D\mathbf k t}(e^{\mathcal O(|\mathbf k|^3)t}-1)\partial^{\alpha-\nu}\chi_1(\mathbf k)
\end{equation}
and
\begin{equation}\label{eq:I(2)}
	I^{(2)}:=\sum_{\nu\le \alpha}\sum_{\tau< \nu}\begin{pmatrix} \alpha \\ \nu \end{pmatrix}\begin{pmatrix} \nu \\ \tau  \end{pmatrix}\partial^\tau e^{-\mathbf k\cdot \mathbf D\mathbf k t}\partial^{\nu-\tau}e^{\mathcal O(|\mathbf k|^3)t}\partial^{\alpha-\nu}\chi_1(\mathbf k).
\end{equation}
By the estimate \eqref{eq:derivatives estimate} in Lemma \ref{lem:derivatives estimate} and since $\chi_1\in C_c^{\infty}(\mathbb R^d)$ and $\mathbf D\in\mathbb R^{d\times d}$ is positive definite, there are constants $c>0$ and $C>0$ such that 
\begin{align*}
	|I^{(1)}|&\le C\sum_{\nu\le \alpha}|\partial^\nu e^{-\mathbf k\cdot \mathbf D\mathbf k t}||e^{\mathcal O(|\mathbf k|^3)t}-1|\nonumber\\
	&\le C\sum_{\nu\le \alpha}\sum_{\{\mathcal I_j:j=1,\dots,r\},r\le |\nu|}|\mathbf k|^{m_1+3}t^{\ell+m_0+m_1+1}e^{-c|\mathbf k|^2t},
\end{align*}
where $m_k\ge 0$ is the cardinality of the set $\{j\in\{1,\dots,r\}:|\mathcal I_j|=|\nu|-k\}$ for $k=0,1$ and $\ell\ge 0$ satisfies
\begin{equation*}
	2\ell< |\nu|-2m_0-m_1
\end{equation*}
and $\{\mathcal I_j:j=1,\dots,r\}$ is any possible partition of the index-set $\mathcal I_\nu$ determined by $\nu$. Thus, we have
\begin{align}\label{eq:L2 bound of I(1)}
	\|I^{(1)}\|_{L^2}^2&\le C\sum_{\nu\le \alpha}\sum_{\{\mathcal I_j:j=1,\dots,r\},r\le |\nu|}\int_{\mathbb R^d}|\mathbf k|^{2(m_1+3)}t^{2(\ell+m_0+m_1+1)}e^{-2c|\mathbf k|^2t}\,\dif \mathbf k\nonumber\\
	&\le C\sum_{\nu\le \alpha}\sum_{\{\mathcal I_j:j=1,\dots,r\},r\le |\nu|}(1+t)^{-\frac d2-1+2m_0+m_1+2\ell}\nonumber\\
	&\le C\sum_{\nu\le \alpha}(1+t)^{|\nu|-\frac d2-1}\le C(1+t)^{|\alpha|-\frac d2-1}
\end{align}
since $|\nu|\le |\alpha|$ for all $\nu\le \alpha$.

Similarly, we can estimate $I^{(2)}$ in \eqref{eq:I(2)}. Since $\chi_1\in C_c^\infty(\mathbb R^d)$ and $\mathbf D\in\mathbb R^{d\times d}$ is positive definite, from \eqref{eq:I(2)} and the estimate \eqref{eq:derivatives estimate} in Lemma \ref{lem:derivatives estimate}, there are constants $c,c'>0$ and $C>0$ such that
\begin{align*}
	|I^{(2)}|&\le C\sum_{\nu\le \alpha}\sum_{\tau< \nu}|\partial^\tau e^{-\mathbf k\cdot \mathbf D\mathbf k t}||\partial^{\nu-\tau}e^{\mathcal O(|\mathbf k|^3)t}|\nonumber\\
	&\le C\sum_{\substack{\nu\le \alpha\\ \tau <\nu}}\sum_{\substack{\{\mathcal I_j:j=1,\dots,r\}\\ \{\mathcal I_j':j=1,\dots,r'\}\\r\le |\tau|, r'\le |\nu-\tau|}}|\mathbf k|^{m_1+m_1'+2m_2'}t^{\ell+\ell'+m_0+m_1+m_0'+m_1'+m_2'}e^{-c|\mathbf k|^2t+c'|\mathbf k|^3t},
\end{align*}
where $m_k\ge 0$ is the cardinality of the set $\{j\in\{1,\dots,r\}:|\mathcal I_j|=|\tau|-k\}$ for $k=0,1$ and $m_k'\ge 0$ is the cardinality of the set $\{j\in\{1,\dots,r'\}:|\mathcal I_j'|=|\nu-\tau|-k\}$ for $k=0,1,2$ and $\ell,\ell'\ge 0$ satisfies
\begin{equation*}
	2\ell< |\tau|-2m_0-m_1,\qquad 3\ell'<|\nu-\tau|-3m_0'-2m_1'-m_2',
\end{equation*}
and $\{\mathcal I_j:j=1,\dots,r\}$ is any possible partition of the index-set $\mathcal I_\tau$ determined by $\tau$ and  $\{\mathcal I_j':j=1,\dots,r'\}$ is any possible partition of the index-set $\mathcal I_{\nu-\tau}$ determined by $\nu-\tau$. Hence, since $|\mathbf k|$ small, we have
\begin{align}\label{eq:L2 bound of I(2)}
	\|I^{(2)}\|_{L^2}^2&\le C\sum_{\substack{\nu\le \alpha\\ \tau <\nu}}\sum_{\substack{\{\mathcal I_j:j=1,\dots,r\}\\ \{\mathcal I_j':j=1,\dots,r'\}\\r\le |\tau|, r'\le |\nu-\tau|}}\int_{\mathbb R^d}|\mathbf k|^{2(m_1+m_1'+2m_2')}t^{2(\ell+\ell'+m_0+m_1+m_0'+m_1'+m_2')}e^{-c|\mathbf k|^2t}\,\dif \mathbf k\nonumber\\
	&\le  C\sum_{\substack{\nu\le \alpha\\ \tau <\nu}}\sum_{\substack{\{\mathcal I_j:j=1,\dots,r\}, r\le |\tau|\\ \{\mathcal I_j':j=1,\dots,r'\},r'\le |\nu-\tau|}}(1+t)^{-\frac d2+2m_0+2m_0'+m_1+m_1'+2\ell+2\ell'}\nonumber\\
	&\le C\sum_{\substack{\nu\le \alpha\\ \tau <\nu}}\sum_{\substack{\{\mathcal I_j:j=1,\dots,r\},r\le |\tau|\\ \{\mathcal I_j':j=1,\dots,r'\},r'\le |\nu-\tau|}}(1+t)^{-\frac d2+|\tau|+|\nu-\tau|-(\ell'+m_0'+m_1'+m_2')}\nonumber\\
	&\le C\sum_{\substack{\nu\le \alpha\\ \tau <\nu}}\sum_{\substack{\{\mathcal I_j:j=1,\dots,r\},r\le |\tau|\\ \{\mathcal I_j':j=1,\dots,r'\},r'\le |\nu-\tau|}}(1+t)^{|\nu|-\frac d2-r'}\le C(1+t)^{|\alpha|-\frac d2-1}
\end{align}
since $|\tau|+|\nu-\tau|=|\nu|\le |\alpha|$ for any $\tau\le \nu\le \alpha$ and $\ell'+m_0'+m_1'+m_2'=r'\ge 1$ by definition and $\tau<\nu$.

We continue to estimate $J$ in \eqref{eq:components I and J for low frequency}. By the Leibniz rule, one has
\begin{equation}\label{eq:derivatives of J}
	\partial^\alpha J=\sum_{\nu\le \alpha}\sum_{\tau\le \nu}\begin{pmatrix} \alpha \\ \nu \end{pmatrix}\begin{pmatrix} \nu \\ \tau  \end{pmatrix}\partial^\tau e^{-\mathbf k\cdot \mathbf D\mathbf k t+\mathcal O(|\mathbf k|^3)t}\partial^{\nu-\tau}\mathcal O(|\mathbf k|)\partial^{\alpha-\nu}\chi_1(\mathbf k)=J^{(1)}+J^{(2)},
\end{equation}
where
\begin{equation}\label{eq:J(1)}
	J^{(1)}:=\sum_{\nu\le \alpha}\begin{pmatrix} \alpha \\ \nu \end{pmatrix}\partial^\nu e^{-\mathbf k\cdot \mathbf D\mathbf k t+\mathcal O(|\mathbf k|^3)t}\mathcal O(|\mathbf k|)\partial^{\alpha-\nu}\chi_1(\mathbf k)
\end{equation}
and
\begin{equation}\label{eq:J(2)}
	J^{(2)}:=\sum_{\nu\le \alpha}\sum_{\tau< \nu}\begin{pmatrix} \alpha \\ \nu \end{pmatrix}\begin{pmatrix} \nu \\ \tau  \end{pmatrix}\partial^\tau e^{-\mathbf k\cdot \mathbf D\mathbf k t+\mathcal O(|\mathbf k|^3)t}\partial^{\nu-\tau}\mathcal O(|\mathbf k|)\partial^{\alpha-\nu}\chi_1(\mathbf k).
\end{equation}
We then begin with $J^{(1)}$. Since $\chi_1\in C_c^\infty(\mathbb R^d)$ and $\mathbf D\in\mathbb R^{d\times d}$ is positive definite, from \eqref{eq:J(1)} and the estimate \eqref{eq:derivatives estimate} in Lemma \ref{lem:derivatives estimate}, there are constants $c,c'>0$ and $C>0$ such that
\begin{align*}
	|J^{(1)}|&\le C\sum_{\nu\le \alpha}|\partial^\nu e^{-\mathbf k\cdot \mathbf D\mathbf k t+\mathcal O(|\mathbf k|^3)t}||\mathcal O(|\mathbf k|)|\nonumber\\
	&\le C\sum_{\nu\le \alpha}\sum_{\{\mathcal I_j:j=1,\dots,r\}}|\mathbf k|^{m_1+1}t^{\ell+m_0+m_1}e^{-c|\mathbf k|^2t+c'|\mathbf k|^3t},
\end{align*}
where $m_k\ge 0$ is the cardinality of the set $\{j\in\{1,\dots,r\}:|\mathcal I_j|=|\nu|-k\}$ for $k=0,1$ and $\ell\ge 0$ satisfies
\begin{equation*}
	2\ell< |\nu|-2m_0-m_1
\end{equation*}
and $\{\mathcal I_j:j=1,\dots,r\}$ is any possible partition of the index-set $\mathcal I_\nu$ determined by $\nu$. Since $|\mathbf k|$ small, it implies that
\begin{align}\label{eq:L2 bound of J(1)}
	\|J^{(1)}\|_{L^2}&\le C\sum_{\nu\le \alpha}\sum_{\{\mathcal I_j:j=1,\dots,r\}}\int_{\mathbb R^d}|\mathbf k|^{2(m_1+1)}t^{2(\ell+m_0+m_1)}e^{-c|\mathbf k|^2t}\,\dif \mathbf k\nonumber\\
	&\le C\sum_{\nu\le \alpha}\sum_{\{\mathcal I_j:j=1,\dots,r\}}(1+t)^{-\frac d2-1+2\ell+2m_0+m_1}\nonumber\\
	&\le C\sum_{\nu\le \alpha}(1+t)^{|\nu|-\frac d2-1}\le C(1+t)^{|\alpha|-\frac d2-1}
\end{align}
since $|\nu|\le |\alpha|$ for all $\nu\le \alpha$.

We estimate $J^{(2)}$ in \eqref{eq:J(2)}. Since $\chi_1\in C_c^\infty(\mathbb R^d)$ and $\mathbf D\in\mathbb R^{d\times d}$ is positive definite, from \eqref{eq:J(2)} and the estimate \eqref{eq:derivatives estimate} in Lemma \ref{lem:derivatives estimate}, there are constants $c,c'>0$ and $C>0$ such that
\begin{align*}
	|J^{(2)}|&\le C\sum_{\nu\le \alpha}\sum_{\tau< \nu}|\partial^\tau e^{-\mathbf k\cdot \mathbf D\mathbf k t+\mathcal O(|\mathbf k|^3)t}||\partial^{\nu-\tau}\mathcal O(|\mathbf k|)|\nonumber\\
	&\le C\sum_{\nu\le \alpha}\sum_{\tau< \nu}\sum_{\{\mathcal I_j:j=1,\dots,r\},r\le |\tau|}|\mathbf k|^{m_1}t^{\ell+m_0+m_1}e^{-c|\mathbf k|^2t+c'|\mathbf k|^3t},
\end{align*}
where $m_k\ge 0$ is the cardinality of the set $\{j\in\{1,\dots,r\}:|\mathcal I_j|=|\tau|-k\}$ for $k=0,1$ and $\ell\ge 0$ satisfies
\begin{equation*}
	2\ell< |\tau|-2m_0-m_1
\end{equation*}
and $\{\mathcal I_j:j=1,\dots,r\}$ is any possible partition of the index-set $\mathcal I_\tau$ determined by $\tau$. Thus, we have
\begin{align}\label{eq:L2 bound of J(2)}
	\|J^{(2)}\|_{L^2}^2 &\le C\sum_{\nu\le \alpha}\sum_{\tau< \nu}\sum_{\{\mathcal I_j:j=1,\dots,r\},r\le |\tau|}\int_{\mathbb R^d}|\mathbf k|^{2m_1}t^{2(\ell+m_0+m_1)}e^{-c|\mathbf k|^2t+c'|\mathbf k|^3t}\,\dif \mathbf k\nonumber\\
	&\le C\sum_{\nu\le \alpha}\sum_{\tau< \nu}\sum_{\{\mathcal I_j:j=1,\dots,r\},r\le |\tau|}(1+t)^{-\frac d2+2\ell+2m_0+m_1}\nonumber\\
	&\le C\sum_{\nu\le \alpha}\sum_{\tau< \nu}(1+t)^{|\tau|-\frac d2}\le C\sum_{\nu\le \alpha}(1+t)^{|\nu|-\frac d2-1}\le C(1+t)^{|\alpha|-\frac d2-1}
\end{align}
since $|\tau|\le |\nu|-1$ for all $\tau<\nu$ and $|\nu|\le |\alpha|$ for all $\nu\le \alpha$.

Therefore, from \eqref{eq:decomposition of solution for low frequency 1 into I and J}, \eqref{eq:components I and J for low frequency} and \eqref{eq:derivatives of I} - \eqref{eq:L2 bound of J(2)}, one has
\begin{equation}\label{eq:L2 bound of derivative of I+J}
	\|\partial^\alpha((\hat \Gamma_t^{(1)}-\hat \Phi_t)\chi_1)\|_{L^2}\le C(1+t)^{\frac{|\alpha|}{2}-\frac d4-\frac 12}, \qquad \forall t>0, \alpha\in\mathbb N^d.
\end{equation}
Let $s\in \mathbb Z_+$ satisfying $s>d/2$, then by the Carlson--Beurling inequality \eqref{eq:Carlson-Beurling Hs} in Lemma \ref{lem:Carlson-Beurling} and \eqref{eq:L2 bound of derivative of I+J}, we obtain
\begin{align}\label{eq:Mp bound of Gamma_t(1)-Phi_t}
	\|(\hat \Gamma_t^{(1)}-\Phi_t)\chi_1\|_{M_p}&\le \|(\hat \Gamma_t^{(1)}-\Phi_t)\chi_1\|_{L^2}^{1-\frac{d}{2s}}(\sum_{|\alpha|=s}\|\partial^\alpha((\hat \Gamma_t^{(1)}-\Phi_t)\chi_1) \|_{L^2})^{\frac{d}{2s}}\nonumber\\
	&\le C(1+t)^{-(\frac d4+\frac 12)(1-\frac{d}{2s})+(\frac{s}{2}-\frac d4-\frac 12)\frac{d}{2s}}\le C(1+t)^{-\frac 12},
\end{align}
for all $1\le p\le \infty$ and $t>0$.

It follows from \eqref{eq:Mp bound of Gamma_t(1)-Phi_t} and the definition of the $M_p$-norm that for any $u_0\in L^p(\mathbb R^d)$ for $1\le p\le \infty $, one has
\begin{align}\label{eq:Lp-Lp bound}
	\|\mathcal F^{-1}((\hat \Gamma_t(\mathbf k)P_0(i\mathbf k)-\hat \Phi_t(\mathbf k))\chi_1(\mathbf k))*u_0\|_{L^p}& \le \|(\hat \Gamma_t^{(1)}-\hat \Phi_t)\chi_1\|_{M_p}\|u_0\|_{L^p}\nonumber\\
	&\le C(1+t)^{-\frac 12}\|u_0\|_{L^p}, \qquad \forall t>0.
\end{align}

We estimate the remain parts. Recall $\hat \Gamma_t^{(2)}$ in \eqref{eq:decomposition of solution for low frequency 2} and the estimate \eqref{eq:point-wise estimate for Gamma_t(2)}, there are constants $c>0$ and $C>0$ such that
\begin{equation}\label{eq:point-wise estimate Gamma2 new}
	|\hat \Gamma_t^{(2)}(\mathbf k)\chi_1(\mathbf k)|\le Ce^{-ct}|\chi_1(\mathbf k)|.
\end{equation}
Thus, by the Parseval identity, one has
\begin{align}\label{eq:L2-L2 bound of Gamma_t(2)}
	\|\mathcal F^{-1}(\hat\Gamma_t(\mathbf k)(I-P_0(\mathbf k))\chi_1(\mathbf k))*u_0\|_{L^2}=\|\hat \Gamma_t^{(2)}\chi_1\hat u_0\|_{L^2}&\le Ce^{-ct}\|\chi_1\|_{L^\infty}\|\hat u_0\|_{L^2}\nonumber\\
	&\le Ce^{-ct}\|u_0\|_{L^2},\qquad \forall t>0.
\end{align}

Moreover, if $\Gamma_t$ has compact support contained in $\{(\mathbf x,t)\in\mathbb R^d\times \mathbb R:|\mathbf x/t|\le C\}$ for some constant $C>0$. We then obtain from \eqref{eq:point-wise estimate Gamma2 new} and the Young inequality that there is $c',c>0$ and $C>0$ such that for $1\le p\le \infty$, one has
\begin{align}\label{eq:L2-L2 bound of Gamma_t(2) extend}
	\|\mathcal F^{-1}(\hat\Gamma_t(\mathbf k)(I-P_0(\mathbf k))\chi_1(\mathbf k))*u_0\|_{L^p}&\le C\|\mathcal F^{-1}(\hat\Gamma_t(\mathbf k)(I-P_0(\mathbf k))\chi_1(\mathbf k))\|_{L^1}\|u_0\|_{L^p}\nonumber\\
	&\le C\Bigl(\int_{|\mathbf x|\le Ct}\Bigl|\int_{|\mathbf k|<\varepsilon}e^{i\mathbf x\cdot \mathbf k}\hat\Gamma_t^{(2)}(\mathbf k)\chi_1(\mathbf k)\,\dif \mathbf k\Bigr|\,\dif \mathbf x\Bigr) \|u_0\|_{L^p}\nonumber\\
	&\le Ce^{-c't}t^d\|u_0\|_{L^p}\le Ce^{-ct}\|u_0\|_{L^p},\qquad \forall t\ge 1.
\end{align}

\vskip.25cm
{\it Step 3. Interpolation.}
\vskip.25cm
Finally, consider the interpolation inequality, from \eqref{eq:Linfty-L1 bound} and \eqref{eq:Lp-Lp bound}, we obtain \eqref{eq:low frequency estimate 1}, namely for $1\le q\le p\le \infty$ and $t>0$, one has
\begin{equation*}
	\|\mathcal F^{-1}((\hat \Gamma_t(\mathbf k)P_0(i\mathbf k)-\hat \Phi_t(\mathbf k))\chi_1(\mathbf k))*u_0\|_{L^p}\le C(1+t)^{-\frac d2(\frac 1q-\frac1p)-\frac 12}\|u_0\|_{L^q}.
\end{equation*}
We also obtain \eqref{eq:low frequency estimate 3} from \eqref{eq:Linfty-L1 bound Gamma2} and \eqref{eq:L2-L2 bound of Gamma_t(2)} {\it i.e.} for $1\le q\le 2\le p\le \infty$, we have
\begin{equation}\label{eq:temp}
	\|\mathcal F^{-1}(\hat \Gamma_t(\mathbf k)(I-P_0(i\mathbf k))\chi_1(\mathbf k))*u_0\|_{L^p}\le Ce^{-ct}\|u_0\|_{L^q}.
\end{equation}
Moreover, if $\Gamma_t$ has support in $\{(\mathbf x,t)\in\mathbb R^d\times \mathbb R:|\mathbf x/t|\le C\}$ for some constant $C>0$, then from \eqref{eq:Linfty-L1 bound Gamma2} and \eqref{eq:L2-L2 bound of Gamma_t(2) extend}, we also have \eqref{eq:temp} for $1\le q\le p\le \infty$ and $t\ge 1$.
\vskip.25cm
{\it Under the symmetry property $\mathsf S$.}
\vskip.25cm
Moreover, if in addition the condition $\mathsf S$ holds, then for small $\mathbf k$, from \eqref{eq:E(z) for low frequency} - \eqref{eq:the 0-group 2} in Proposition \ref{prop:low-frequency}, one has $\hat \Gamma_t\chi_1=\hat \Gamma_t^{(1)}\chi_1+\hat \Gamma_t^{(2)}\chi_1$ where
\begin{equation}\label{eq:decomposition of solution for low frequency 2 1}
	\hat \Gamma_t^{(1)}(\mathbf k)=e^{-\lambda_0(i\mathbf k)t}P_0(i\mathbf k)=e^{-\mathbf k\cdot \mathbf D\mathbf k t+\mathcal O(|\mathbf k|^4)t}(P_0^{(0)}+\mathbf P_0^{(1)}\cdot i\mathbf k+\mathcal O(|\mathbf k|^2)),
\end{equation}
and
\begin{equation}\label{eq:decomposition of solution for low frequency 2 2}
	\hat \Gamma_t^{(2)}(\mathbf k)=\sum_{j=1}^se^{-E_j(i\mathbf k)t}P_j(i\mathbf k)=\sum_{j=1}^se^{-\lambda_j^{(0)}t}e^{-N_j^{(0)}t+\mathcal O(|\mathbf k|)t}(P_j^{(0)}+\mathcal O(|\mathbf k|)),
\end{equation}
where $\mathbf c\in\mathbb R^d$ and $\mathbf D\in\mathbb R^{d\times d}$ is positive definite given by \eqref{eq:coefficients c and D}, $P_0^{(0)}$ is the eigenprojection associated with $0\in\sigma(B)$, $\mathbf P_0^{(1)}\in(\mathbb R^{n\times n})^d$ is in \eqref{eq:coefficient P01 of the total projection P0}, and $\lambda_j^{(0)}\in\sigma(B)\backslash \{0\}$, $\re \lambda_j^{(0)}>0$, with the associated eigenprojection $P_j^{(0)}$ and eigennilpotent $N_j^{(0)}$ for $j\in\{1,\dots,s\}$ and $s$ is the cardinality of $\sigma(B)\backslash \{0\}$.

Hence, consider
\begin{equation*}
	\hat \Psi_t(\mathbf k)=e^{-\mathbf k\cdot \mathbf D\mathbf k t}(P_0^{(0)}+\mathbf P_0^{(1)}\cdot i\mathbf k),
\end{equation*}
one has
\begin{equation}\label{eq:decomposition of solution for low frequency 1 into I and J symmetry}
	(\hat \Gamma_t^{(1)}(\mathbf k)-\hat \Psi_t(\mathbf k))\chi_1(\mathbf k)=I+J,
\end{equation}
where
\begin{equation}\label{eq:components I and J for low frequency symmetry}
	I:=e^{-\mathbf k\cdot \mathbf D\mathbf k t}(e^{\mathcal O(|\mathbf k|^4)t}-1)(P_0^{(0)}+\mathbf P_0^{(1)}\cdot i\mathbf k)\chi_1(\mathbf k),\quad J:=e^{-\mathbf k\cdot \mathbf D\mathbf k t+\mathcal O(|\mathbf k|^4)t}\mathcal O(|\mathbf k|^2)\chi_1(\mathbf k).
\end{equation}
The estimates are then similar to the previous case. We omit the details. We thus obtain for $1\le q\le p\le\infty$ and $t>0$ that
\begin{equation*}
	\|\mathcal F^{-1}((\hat \Gamma_t(\mathbf k)P_0(\mathbf k)-\hat \Psi_t(\mathbf k))\chi_1(\mathbf k))*u_0\|_{L^p}\le C(1+t)^{-\frac d2(\frac 1q-\frac1p)-1}\|u_0\|_{L^q}.
\end{equation*}
The proof is done since the others are also similar to before.
\end{proof}

\subsection{Intermediate-frequency analysis}
We consider the intermediate-frequency part by considering $\hat \Gamma_t\chi_2$, $\hat \Phi_t\chi_2$ and $\hat \Psi_t\chi_2$, where $\hat \Gamma_t$, $\hat \Phi_t$ and $\hat \Psi_t$ are given by \eqref{eq:Fourier transform of Gamma kernel} and \eqref{eq:Fourier transform of Phi and Psi kernel} respectively. One has the following.
\begin{proposition}[Intermediate-frequency estimate]\label{prop:intermediate-frequency estimates}
If the condition $\mathsf D$ holds, then for $1\le q\le 2\le p\le \infty$, there are constants $c>0$ and $C>0$ such that
\begin{equation}\label{eq:intermediate frequency estimate 1}
	\|\mathcal F^{-1}(\hat \Gamma_t(\mathbf k)\chi_2(\mathbf k))*u_0\|_{L^p}\le Ce^{-ct}\|u_0\|_{L^q},\qquad \forall t>0.
\end{equation}
Moreover, \eqref{eq:intermediate frequency estimate 1} holds for $1\le q\le p\le\infty$ and $t\ge 1$ if $\Gamma_t$ has compact support contained in $\{(\mathbf x,t)\in\mathbb R^d\times \mathbb R:|\mathbf x/t|\le C\}$ for some constant $C>0$.
\end{proposition}
\begin{proof}
Recall $E(i\mathbf k)=B+A(i\mathbf k)$ in \eqref{eq:operators E and A} for $\mathbf k\in\mathbb R^d$. We consider $\hat \Gamma_t$ where $\hat \Gamma_t(\mathbf k)=e^{-E(i\mathbf k)t}$. Since the condition $\mathsf D$ holds, $\re \lambda(i\mathbf k)>0$ for any eigenvalue $\lambda(i\mathbf k)$ of $E(i\mathbf k)$ and $\mathbf k\ne \mathbf 0\in\mathbb R^d$. Thus, the operator $e^{-E(i\mathbf k)}$ has the spectral radius $\textrm{rad}(e^{-E(i\mathbf k)})<1$ for almost everywhere. It follows from the Householder theorem in \cite{serre10} that there is an induced norm such that
\begin{equation*}
	0<\varphi:=\esssup_{\mathbb R^d}|e^{-E(i\mathbf k)}|<1.
\end{equation*}
Then, for $t>0$ with integer part $m$, since $\log \varphi<0$, there are $c,C>0$ such that one has
\begin{align}\label{eq:point-wise estimate Gamma_tchi(2)}
	|\hat \Gamma_t(\mathbf k)\chi_2(\mathbf k)|=|e^{-E(i\mathbf k)t}\chi_2(\mathbf k)|&\le |e^{-E(i\mathbf k)}|^m|e^{-E(i\mathbf k)(t-m)}||\chi_2(\mathbf k)|\nonumber\\
	&\le \varphi^me^{|E(i\mathbf k)|}|\chi_2(\mathbf k)|\nonumber\\
	&\le \varphi^{-1} e^{(m+1)\log \varphi}e^{|E(i\mathbf k)|}|\chi_2(\mathbf k)|\le Ce^{-ct}e^{|E(i\mathbf k)|}|\chi_2(\mathbf k)|.
\end{align}

We study the $L^\infty$-$L^1$ estimate. By the Young inequality and from \eqref{eq:point-wise estimate Gamma_tchi(2)}, there are constants $c>0$ and $C>0$ such that for $t>0$, we have
\begin{align}\label{eq:Linfty-L1 bound 1}
	\|\mathcal F^{-1}(\hat \Gamma_t(\mathbf k)\chi_2(\mathbf k))*u_0\|_{L^\infty}&\le C\|\mathcal F^{-1}(\hat \Gamma_t(\mathbf k)\chi_2(\mathbf k))\|_{L^\infty}\|u_0\|_{L^1}\nonumber\\
	&\le C\|\hat \Gamma_t\chi_2\|_{L^1}\|u_0\|_{L^1}\le Ce^{-ct}\|u_0\|_{L^1}.
\end{align}

We prove the $L^2$-$L^2$ estimate. It follows from the Parseval identity and the estimate \eqref{eq:point-wise estimate Gamma_tchi(2)} that for $t>0$, one has
\begin{equation}\label{eq:Lp-Lp bound 1}
	\|\mathcal F^{-1}(\hat \Gamma_t(\mathbf k)\chi_2(\mathbf k))*u_0\|_{L^2}\le  C\|\hat \Gamma_t\chi_2\|_{L^\infty}\|\hat u_0\|_{L^2}\le Ce^{-ct}\|u_0\|_{L^2}
\end{equation}
for some constants $c>0$ and $C>0$.

Moreover, if $\Gamma_t$ has compact support contained in $\{(\mathbf x,t)\in\mathbb R^d\times \mathbb R:|\mathbf x/t|\le C\}$ for some constant $C>0$. From \eqref{eq:point-wise estimate Gamma_tchi(2)} and the Young inequality, there are $c',c>0$ and $C>0$ such that for $1\le p\le \infty$, one has
\begin{align}\label{eq:Lp-Lp bound of Gamma_tchi(2) extend}
	\|\mathcal F^{-1}(\hat\Gamma_t(\mathbf k)\chi_2(\mathbf k))*u_0\|_{L^p}&\le C\|\mathcal F^{-1}(\hat\Gamma_t(\mathbf k)\chi_2(\mathbf k))\|_{L^1}\|u_0\|_{L^p}\nonumber\\
	&\le C\Bigl(\int_{|\mathbf x|\le Ct}\Bigl|\int_{\varepsilon\le |\mathbf k|\le \rho}e^{i\mathbf x\cdot \mathbf k}\hat\Gamma_t(\mathbf k)\chi_2(\mathbf k)\,\dif \mathbf k\Bigr|\,\dif \mathbf x\Bigr) \|u_0\|_{L^p}\nonumber\\
	&\le Ce^{-c't}t^{d}\|u_0\|_{L^p}\le Ce^{-ct}\|u_0\|_{L^p},\qquad \forall t\ge 1.
\end{align}

We finish the proof of \eqref{eq:intermediate frequency estimate 1} by applying the interpolation inequality and by using the $L^\infty$-$L^1$ estimate \eqref{eq:Linfty-L1 bound 1}, the $L^2$-$L^2$ estimate \eqref{eq:Lp-Lp bound 1} and the $L^p$-$L^p$ estimates \eqref{eq:Lp-Lp bound of Gamma_tchi(2) extend}.
\end{proof}

Moreover, we have the $L^p$-$L^q$ estimate for $\mathcal F^{-1}(\hat \Phi_t(\mathbf k)\chi_2(\mathbf k))*u_0$ and $\mathcal F^{-1}(\hat \Psi_t(\mathbf k)\chi_2(\mathbf k))*u_0$ for $1\le q\le p\le \infty$ as the following.
\begin{proposition}\label{prop:intermediate-frequency estimates 2}
If the conditions $\mathsf B$ and $\mathsf D$ hold, then for $1\le q\le p\le \infty$, there are constants $c>0$ and $C>0$ such that for $t\ge 1$, one has
\begin{equation}\label{eq:Lp-Lq estimate for Phi_tchi(2)}
	\|\mathcal F^{-1}(\hat \Phi_t(\mathbf k)\chi_2(\mathbf k))*u_0\|_{L^p}\le Ce^{-ct}\|u_0\|_{L^q}.
\end{equation}
Similarly, we have
\begin{equation}\label{eq:Lp-Lq estimate for Psi_tchi(2)}
	\|\mathcal F^{-1}(\hat \Psi_t(\mathbf k)\chi_2(\mathbf k))*u_0\|_{L^p}\le Ce^{-ct}\|u_0\|_{L^q}.
\end{equation}
\end{proposition}
\begin{proof}
We estimate $\mathcal F^{-1}(\hat \Phi_t(\mathbf k)\chi_2(\mathbf k))*u_0$ and the other is similar. Recall that
\begin{equation*}
	\hat \Phi_t(\mathbf k)=e^{-\mathbf c\cdot i\mathbf kt-\mathbf k\cdot \mathbf D\mathbf kt}P_0^{(0)},
\end{equation*}
where $\mathbf c\in\mathbb R^d$ and $\mathbf D\in\mathbb R^{d\times d}$ is positive definite given by \eqref{eq:coefficients c and D} under the assumptions $\mathsf B$ and $\mathsf D$.

Since we can assume that $\mathbf c=\mathbf 0$ and since $\textrm{supp}\,\chi_2\subseteq \{\mathbf k\in\mathbb R^d:\varepsilon\le |\mathbf k|\le \rho\} $ for some $\varepsilon, \rho>0$, by the Young inequality, there are constants $c,c'>0$ and $C>0$ such that we have the $L^\infty$-$L^1$ estimate
\begin{align}\label{eq:L^infty-L1 estimate (2)}
	\|\mathcal F^{-1}(\hat \Phi_t(\mathbf k)\chi_2(\mathbf k))*u_0\|_{L^\infty}&\le C\|\mathcal F^{-1}(e^{-\mathbf k\cdot \mathbf D\mathbf kt}P_0^{(0)}\chi_2(\mathbf k))\|_{L^\infty}\|u_0\|_{L^1}\nonumber\\
	&\le Ce^{-c't}\|e^{-\frac {c'}{2}|\cdot |^2t}\|_{L^1}\|u_0\|_{L^1}\le Ce^{-ct}\|u_0\|_{L^1},\qquad \forall t>0.
\end{align}

We study the $L^p$-$L^p$ estimate for $1\le p\le \infty$. Primarily, we have
\begin{equation*}
	\|e^{-\mathbf k\cdot \mathbf D\mathbf kt}P_0^{(0)}\chi_2(\mathbf k)\|_{L^2}\le Ce^{-c't}\|e^{-\frac {c'}{2}|\cdot |^2t}\|_{L^2}\le Ce^{-ct}, \qquad \forall t>0.
\end{equation*}

Let $\alpha\in \mathbb N^d$, by the Leibniz formula, \eqref{eq:derivatives estimate parabolic previous} and Remark \ref{rem:derivatives estimate parabolic previous}, one has
\begin{align*}
	|\partial^\alpha ( e^{-\mathbf k\cdot \mathbf D\mathbf kt}P_0^{(0)}\chi_2(\mathbf k))|&\le C\sum_{\nu \le \alpha}|\partial^\nu e^{-\mathbf k\cdot \mathbf D\mathbf kt}||\partial^{\alpha-\nu}\chi_2(\mathbf k)|\nonumber\\
	&\le C\sum_{\nu\le \alpha}\sum_{\substack{\{\mathcal I_j:j=1,\dots,r\},r\le |\alpha|\\
	1\le |\mathcal I_j|\le 2}}|\mathbf k|^mt^{m+\ell}e^{-c'|\mathbf k|^2t}|\partial^{\alpha-\nu}\chi_2(\mathbf k)|,
\end{align*}
where $\{\mathcal I_j:j=1,\dots,r\}$ is any possible partition of the index-set $\mathcal I_\alpha$ determined by $\alpha$ and $m+2\ell=|\alpha|$. Hence, we also have
\begin{equation*}
	\|\partial^\alpha (e^{-\mathbf k\cdot \mathbf D\mathbf kt}P_0^{(0)}\chi_2(\mathbf k))\|_{L^2}\le Ce^{-ct}, \qquad\forall t\ge 1,\alpha\in\mathbb N^d.
\end{equation*}

Therefore, by the Carlson--Beurling inequality \eqref{eq:Carlson-Beurling Hs} in Lemma \ref{lem:Carlson-Beurling}, one obtains
\begin{equation}\label{eq:Lp-Lp estimate (2)}
	\|e^{-\mathbf k\cdot \mathbf D\mathbf kt}P_0^{(0)}\chi_2(\mathbf k)\|_{M_p}\le Ce^{-ct},\qquad \forall 1\le p\le \infty,t\ge 1.
\end{equation}

Finally, by the interpolation inequality, the estimates \eqref{eq:L^infty-L1 estimate (2)} and \eqref{eq:Lp-Lp estimate (2)}, we obtain \eqref{eq:Lp-Lq estimate for Phi_tchi(2)}. The proof is done.
\end{proof}

\subsection{High-frequency analysis}
The aim of this part is to give an $L^2$-$L^2$ estimate of the high-oscillation part of $\Gamma_t$, which is $\hat\Gamma_t\chi_3$ in the Fourier space, where $\hat \Gamma_t$ is given by \eqref{eq:Fourier transform of Gamma kernel}.
\begin{proposition}[High-frequency estimate]\label{prop:high-frequency estimates}
If the conditions $\mathsf A$, $\mathsf R$ and $\mathsf D$ hold, then there are constants $c>0$ and $C>0$ such that one has the estimate
\begin{equation*}
	\|\mathcal F^{-1}(\hat \Gamma_t(\mathbf k)\chi_3(\mathbf k))*u_0\|_{L^2}\le Ce^{-ct}\|u_0\|_{L^2},\qquad \forall t>0.
\end{equation*}
\end{proposition}
\begin{proof}
Under the assumptions $\mathsf A$, $\mathsf R$ and $\mathsf D$, for almost everywhere and large $\mathbf k\in\mathbb R^d$, from \eqref{eq:E(z) for high frequency} - \eqref{eq:the total projection Pijm0}, we have
\begin{equation*}
	\hat \Gamma_t(\mathbf k)\chi_3(\mathbf k)=R\sum_{j=1}^r\sum_{m=1}^{s_j}e^{-\alpha_j(i\mathbf k)t}e^{-\beta_{jm}t}e^{\Theta_{jm}^{(0)}t+\mathcal O(|\mathbf k|^{-1})t}(\Pi_{jm}^{(0)}+\mathcal O(|\mathbf k|^{-1}))R^{-1}\chi_3(\mathbf k),
\end{equation*}
where $R$ is the invertible matrix satisfying the conditions $\mathsf A$ and $\mathsf R$, $\alpha_j(i\mathbf k)=i|\mathbf k|\nu_{[j]}(\mathbf k/|\mathbf k|)$ for $\nu_{[j]}$ is given by \eqref{eq:eigenvalues of A}, $\beta_{jm}$ with $\re \beta_{jm}>0$ is the $m$-th nonzero eigenvalue of $\Pi_j^{(0)}R^{-1}BR\Pi_j^{(0)}$ with the associated eigenprojection $\Pi_{jm}^{(0)}$ and eigennilpotent $\Theta_{jm}^{(0)}$, where $\Pi_j^{(0)}$ is in \eqref{eq:the eigenprojection Pij0}.

Thus, by the Householder theorem in \cite{serre10}, for all $\varepsilon>0$, there is an induced norm such that $|\Theta_{jm}^{(0)}|\le \varepsilon$ and due to the fact that every norms in finite-dimensional space are equivalent, one deduces that since $|\mathbf k|$ large and $\re \beta_{jm}^{(0)}>0$ for all $j\in\{1,\dots,r\}$ and $m\in\{1,\dots,s_j\}$, there are constants $c,c'>0$ and $C>0$ such that
\begin{align*}
	|\hat \Gamma_t(\mathbf k)\chi_3(\mathbf k)|&\le C\sum_{j=1}^r\sum_{m=1}^{s_j}e^{-\re \beta_{jm}t}e^{\varepsilon t+c'|\mathbf k|^{-1}t}(1+|\mathbf k|^{-1})|\chi_3(\mathbf k)|\nonumber\\
	&\le Ce^{-ct}(1+|\mathbf k|^{-1})|\chi_3(\mathbf k)|.
\end{align*}

Therefore, by the Parseval identity, we have
\begin{align*}
	\|\mathcal F^{-1}(\hat \Gamma_t(\mathbf k)\chi_3(\mathbf k))*u_0\|_{L^2}&=\|\hat \Gamma_t(\mathbf k)\chi_3(\mathbf k)\hat u_0(\mathbf k)\|_{L^2}\nonumber\\
	&\le C\|\hat \Gamma_t(\mathbf k)\chi_3(\mathbf k)\|_{L^\infty}\|\hat u_0\|_{L^2}\nonumber\\
	&\le Ce^{-ct}\|(1+|\cdot|^{-1})\chi_3\|_{L^\infty}\|u_0\|_{L^2}\le Ce^{-ct}\|u_0\|_{L^2}
\end{align*}
for some constants $c,C>0$ and for all $t>0$. We finish the proof.
\end{proof}

Moreover, we have the $L^p$-$L^q$ estimate for $\mathcal F^{-1}(\hat \Phi_t(\mathbf k)\chi_3(\mathbf k))*u_0$ and $\mathcal F^{-1}(\hat \Psi_t(\mathbf k)\chi_3(\mathbf k))*u_0$ for $1\le q\le p\le \infty$ as the following.
\begin{proposition}\label{prop:high-frequency estimates 2}
If the conditions $\mathsf B$ and $\mathsf D$ hold, then for $1\le q\le p\le \infty$, there are constants $c>0$ and $C>0$ such that for $t\ge 1$, one has
\begin{equation}\label{eq:Lp-Lq estimate for Phi_tchi(3)}
	\|\mathcal F^{-1}(\hat \Phi_t(\mathbf k)\chi_3(\mathbf k))*u_0\|_{L^p}\le Ce^{-ct}\|u_0\|_{L^q}.
\end{equation}
Similarly, we have
\begin{equation}\label{eq:Lp-Lq estimate for Psi_tchi(3)}
	\|\mathcal F^{-1}(\hat \Psi_t(\mathbf k)\chi_3(\mathbf k))*u_0\|_{L^p}\le Ce^{-ct}\|u_0\|_{L^q}.
\end{equation}
\end{proposition}
\begin{proof}
Similarly to the proof of Proposition \ref{prop:intermediate-frequency estimates 2} where $\chi_2$ is substituted by $\chi_3$.
\end{proof}
\section*{Acknowledgments}
The author is grateful to Professor Corrado Mascia for his suggestion and useful comments to improve this paper.
%------------References----------------
\bibliographystyle{elsarticle-num-sort}
\bibliography{Refinedversion}
\end{document}